\documentclass[a4paper]{amsart}
\usepackage{amssymb}
\usepackage{amsmath}
\usepackage{amscd}
\usepackage[all]{xy}

\newtheorem{theorem}{Theorem}
\newtheorem{lemma}{Lemma}
\newtheorem{proposition}{Proposition}
\newtheorem{corollary}{Corollary}

\theoremstyle{definition}
\newtheorem{definition}{Definition}
\newtheorem{remark}{Remark}
\newtheorem{example}{Example}

\newcommand{\p}{{\mathbb P}}

\newcommand{\n}{{\mathbb N}}
\newcommand{\q}{{\mathbb Q}}
\renewcommand{\c}{{\mathbb C}}

\newcommand{\g}{{\mathbb G}}
\newcommand{\z}{{\mathbb Z}}

\renewcommand{\O}{{\mathcal O}}

\newcommand{\da}{{\dashrightarrow}}

\newcommand{\Sec}{\operatorname{Sec}}
\newcommand{\pic}{\operatorname{Pic}}
\newcommand{\length}{\operatorname{length}}
\newcommand{\rk}{\operatorname{rank}}
\newcommand{\sing}{\operatorname{Sing}}

\begin{document}
\title{Special birational transformations of projective spaces}
\author[Alberto Alzati \and Jos\'{e} Carlos Sierra]{Alberto Alzati* \and Jos\'{e} Carlos Sierra**}
\address{Dipartimento di Matematica, Universit\`a degli Studi di Milano\\
via Cesare Saldini 50, 20133 Milano, Italy}
\email{alberto.alzati@unimi.it}

\address{Instituto de Ciencias Matem\'aticas (ICMAT), Consejo Superior de Investigaciones Cient\'{\i}ficas (CSIC), Campus de
Cantoblanco, 28049 Madrid, Spain}
\email{jcsierra@icmat.es}

\thanks{* This work is within the framework of the national research project
``Geometria delle Variet\`a Algebriche" PRIN 2010 of MIUR}

\thanks {** Research supported by the ``Ram\'on y Cajal" contract
RYC-2009-04999, the project MTM2009-06964 of MICINN and the ICMAT ``Severo Ochoa" project SEV-2011-0087 of MINECO}

%\date{\today}
%\keywords{}
\begin{abstract}
Extending some results of Crauder and Katz, and Ein and Shepherd-Barron on special Cremona transformations, we study birational transformations of $\p^r$ onto a prime Fano manifold such that the base locus $X\subset\p^r$ is smooth irreducible and reduced. The main results are a complete classification when $X$ has either dimension $1$ and $2$, or codimension $2$. Partial results are also obtained when $X$ has dimension $3$.
\end{abstract}
\maketitle

\section{Introduction}

In this paper we study birational transformations $\Phi:\p^r\da Z$ with smooth irreducible and reduced base locus $X\subset\p^r$ from the complex projective space onto a prime Fano manifold $Z$, thus extending the study of the classical special Cremona transformations initiated in \cite{s-t1}, \cite{s-t2} and \cite{s-t3}, and more recently and systematically revisited in \cite{c-k} (see also \cite{katz} and \cite{h-k-s}), \cite{e-sb} and \cite{c-k2}. As there are no small contractions involved, these birational transformations are among the most elementary and special links of type II between Mori fibre spaces (in the sense of the Sarkisov program), so maybe they also deserve some attention from this point of view. The main results of the paper are the following:

\begin{theorem}
There are 8 (resp. 13) types of birational transformations of $\p^r$ with smooth irreducible and reduced base locus of dimension $1$ (resp. $2$) onto a prime Fano manifold.
\end{theorem}

This result is explicitly stated in Theorems \ref{thm:n=1} and \ref{thm:n=2}, and it extends the main results of \cite{c-k} (see \cite[Theorems 2.2 and 3.3]{c-k}). On the other hand, when the base locus has dimension $3$ the picture is much more difficult. Note that even the case of special Cremona transformations was left open for $r\in\{6,8\}$ (see \cite[Corollary 1]{c-k2}) and, to the best of the authors' knowledge, it still remains. In our setting, we are able to obtain the classification in many simple cases (see Theorem \ref{thm:n=3 easy}) and we also get the following more interesting result in Theorem \ref{thm:n=3 r=5}:

\begin{theorem}
There are 7 types of birational transformations of $\p^5$ with smooth irreducible and reduced base locus of dimension $3$ onto a prime Fano mani\-fold.
In particular, we construct a new Fano manifold of coindex $4$ and degree $21$.
\end{theorem}

As a consequence, extending one of the main results of \cite{e-sb} (see \cite[Theorem 3.2]{e-sb}), we obtain the following result in Corollary \ref{cor:r=n+2} that shows \emph{a posteriori} an unexpected link between this class of birational transformations and the Hartshorne-Rao modules of the base locus (see Remark \ref{rem:h-r}):

\begin{corollary}
There are 18 types of birational transformations of $\p^r$ with smooth irreducible and reduced base locus of codimension $2$ onto a prime Fano manifold.
\end{corollary}

The proofs of our results are obtained in several steps. In the first step, that we call \emph{numerology}, we obtain the possible numerical invariants of a special birational transformation (see Propositions \ref{prop:num}, \ref{prop:num n=1}, \ref{prop:num n=2} and \ref{prop:num n=3}). In a second step, we make use of the normal bundle $N_{X/\p^r}$ to get the \emph{fundamental formulae} (Propositions \ref{prop:ff}, \ref{prop:ff n=1}, \ref{prop:ff n=2} and \ref{prop:ff n=3}), that combined with some multisecant formulae (see Section \ref{section:n=2}) allow us to obtain, after several computations and geometric arguments, the \emph{maximal list} where all the possible numerical invariants of the embedding $X\subset\p^r$ are collected (Propositions \ref{prop:n=1}, \ref{prop:n=2}, \ref{prop:(4,b)}, \ref{prop:(5,b)} and \ref{prop:(5,1)}). As already happened with the case of special Cremona transformations, this is more complicated for $n=3$. Note that, in our setting, the situation is worse because in some cases we cannot even apply to $Z$ the classification results of Fano manifolds of small coindex obtained by Kobayashi-Ochiai \cite{k-o}, Iskovskikh and Fujita \cite{fuj}, and Mukai \cite{muk} (see Remarks \ref{rem:fano} and \ref{rem:(5,b)}) but, at least, we complete the picture in the same range covered by the classical setting (cf. \cite[Corollary 1]{c-k2}). In a third step, we first screen out the cases (if any) for which there is no $X\subset\p^r$ corresponding to the numerical invariants of the maximal list, and for each one of the remaining cases we classify $X\subset\p^r$. In some cases, this follows from the classification of manifolds of small degree, eventually combined with a \emph{liaison} in codimension $2$, but in some other cases we need to use the Beilinson spectral sequence technique to obtain a resolution of the ideal sheaf from its cohomology (see Lemmas \ref{lemma:12} and \ref{lemma:13}). Once $X\subset\p^r$ and its ideal sheaf are determined, sometimes the resulting $Z$ turns out to be singular \emph{a posteriori} (see Lemmas \ref{lem:n=1 singular}, \ref{lem:n=2 singular}, \ref{lem:(4,b)}, \ref{lem:(5,b)} and \ref{lem:(5,1)}). We include these cases separately because, in any event, we always get a divisorial contraction from the blowing-up of $\p^r$ along $X$ onto $Z$, and hence we get some examples of special birational transformations onto a normal variety of Picard number one with only $\q$-factorial and terminal singularities that might be of some interest by themselves (cf. Remark \ref{rem:sing}). Finally, in the fourth step, we present the main classification results in Theorems \ref{thm:n=1}, \ref{thm:n=2} and \ref{thm:n=3 r=5}, where the smoothness of $Z$ follows from a more general construction of series of birational transformations introduced in Section \ref{section:ex}. We point out that even if we find a lot of special birational transformations of $\p^r$ along the way, they all belong to the short list of series of examples given in Section \ref{section:ex}. In particular, in all the cases $Z$ turns out to be either $\p^r$, or a hypersurface in $\p^{r+1}$, or a complete intersection of quadric hypersurfaces, or a linear section of a Grassmannian of lines, or a prime Fano fivefold of coindex $4$ and degree $21$ described in Example \ref{ex:degree21} that, as far as we know, appears to be new.

We would like to mention that a similar extension of the classification of the quadro-quadric special Cremona transformations \cite[Theorem 2.6]{e-sb}, which is the other main result of that paper, is given in \cite{alzati-sierra} by means of techniques which are not present in this note (namely, the study of rational curves on Fano manifolds in the framework of Mori Theory applied to the study of secant defective manifolds).

\section{Preliminaries and first results}

Let $f_0,\dots,f_{\alpha}\in \c[X_0,\dots,X_r]$ be homogeneous
polynomials of degree $a\geq 2$. Let $\Phi:\p^r\da\p^{\alpha}$ be
the corresponding rational map, and let $Z:=\Phi(\p^r)$. We assume
that $\Phi:\p^r\da Z$ is a birational map and that
$Z\subset\p^{\alpha}$ is a manifold with cyclic Picard group
generated by the hyperplane section $\O_Z(1)$ of
$Z\subset\p^{\alpha}$. In this setting $Z$ is a Fano manifold, as it
is covered by rational curves of degree $a$. Let $i:=i(Z)\in\n$ denote the
index of $Z$, that is, $-K_Z=i\cdot\O_Z(1)$. Let $\Psi:Z\da\p^r$ be the
inverse of $\Phi:\p^r\da Z$. Let $X\subset\p^r$ and $Y\subset Z$
denote the the base (also called fundamental) locus of $\Phi$ and $\Psi$, respectively. We
assume that $X\subset\p^r$ is a smooth irreducible and reduced
scheme. We point out that $X\subset\p^r$ is allowed to be degenerate, i.e. contained in a hyperplane of $\p^r$. According to the classical terminology of Cremona
transformations, we say that $\Phi:\p^r\da Z$ is a \emph{special
birational transformation of $\p^r$}. Let $n:=\dim(X)$ and
$m:=\dim(Y)$, so in particular $1\leq n\leq r-2$ and $0\leq m\leq r-2$. Let $W$ denote the blowing-up of $\p^r$ along $X$,
with projection maps $\sigma:W\to\p^r$ and $\tau:W\to Z$ such that
$\tau=\Phi\circ\sigma$. Let $H:=\sigma^*\O_{\p^r}(1)$ and
$H_Z:=\tau^*\O_Z(1)$. Let $E$ be the exceptional divisor of
$\sigma:W\to\p^r$, and let $E_Z:=\tau^{-1}(Y)$ (scheme
theoretically). We point out that $E_Z$ is irreducible (see
\cite[Proposition 1.3]{e-sb}). Furthermore, a straightforward
generalization of \cite[Proposition 2.1]{e-sb} yields:

\begin{proposition}\label{prop:K}
Notation as above:
\begin{enumerate}
\item[(i)] $E_Z$ is reduced;
\item[(ii)] $H_Z=aH-E$ and $H=bH_Z-E_Z$ for some integer $b\geq1$;
\item[(iii)] $\pic(W)=\z[H]\oplus\z[E]=\z[H_Z]\oplus\z[E_Z]$;
\item[(iv)] $K_W=-(r+1)H+(r-n-1)E=-iH_Z+(r-m-1)E_Z$.
\end{enumerate}
\end{proposition}

Let $T:=\sigma^*(L)\subset W$ and $T_Z:=\tau^*(L_Z)\subset W$, where
$L\subset\p^r$ is a general line and $L_Z\subset Z$ is the curve obtained by intersecting $Z$ with $r-1$ general hyperplanes of $\p^{\alpha}$, respectively. Note that $a=\deg(\Phi(L))=T\cdot H_Z$ and that $bz=\deg(\Psi(L_Z))=T_Z\cdot H$, where $z:=\deg(Z)$.

\begin{definition}
In this setting, we say that $\Phi:\p^r\da Z\subset\p^{\alpha}$ is a \emph{special
birational transformation of $\p^r$ of type $(a,b)$}.
\end{definition}

\begin{remark}
Without loss of generality, we will assume throughout the paper that $f_0,\dots,f_{\alpha}$ is a basis of $H^0({\mathcal I}_X(a))$ or, equivalently, that $Z\subset\p^{\alpha}$ is non-degenerate and linearly normal.
\end{remark}

A line $L\subset\p^r$ is a $k$-secant line of $X$ if $L\not\subset
X$ and $\length(X\cap L)\geq k$. Let $\Sec_k(X)\subset\p^r$ denote
the closure of the union of all the $k$-secant lines of $X$. The
useful remark given in \cite[Proposition 2.3]{e-sb} also holds in our
setting:

\begin{proposition}\label{prop:sec}
Notation as above:
\begin{enumerate}
\item[(i)] $\sigma(E_Z)=\Sec_a(X)\subset\p^r$ is a hypersurface of degree $ab-1$;
\item[(ii)] Let $p\in\sigma(E_Z)$ be a general point. Then the union of all the
$a$-secant lines of $X$ through $p$ is an $(r-m-1)$-plane that
intersects $X$ in a hypersurface of degree $a$.
\end{enumerate}
\end{proposition}

\begin{remark}\label{rem:sec}
We point out that moreover $X\subset\Sec_a(X)$. As $\sigma(E_Z)=\Sec_a(X)$, it is enough to show that the intersection in $W$ of $E_Z$ and every fibre of $E$ over $X$ is positive. For any $x\in X$, let $\xi$ be a curve contained in $\p^{r-n-1}\cong\sigma^{-1}(x)\subset E$. Then $H\cdot\xi=0$ and $H_Z\cdot\xi>0$. As $H=bH_Z-E_Z$, we get $E_Z\cdot\xi=bH_Z\cdot\xi>0$.
\end{remark}

Note that, in particular, Proposition \ref{prop:sec}(i) yields $r\leq 2n+2$, and equality holds only for $a=2$. Furthermore, there are some stronger restrictions on the numerical invariants $\{a,b,n,m,i\}$ of a birational transformation $\Phi:\p^r\da Z$. The following result will be very important in what follows (cf. \cite[Lemma 2.4]{e-sb}):

\begin{proposition}[Numerology]\label{prop:num}
Let $\Phi:\p^r\da Z$ be a special birational transformation of type
$(a,b)$. Then:
\begin{enumerate}
%\item[(i)] $T\cdot E_Z=ab-1$ and $T_Z\cdot E=(ab-1)z$
\item[(i)] $i=(r+1)b-(r-n-1)(ab-1)$
\item[(ii)] $r+1=ia-(r-m-1)(ab-1)$
\item[(iii)] $a=(m+2)/(r-n-1)$
\item[(iv)] $b=(n+2+i-r-1)/(r-m-1)$
\end{enumerate}
\end{proposition}

\begin{proof}
Since $E_Z=bH_Z-H$ and $E=aH-H_Z$, we deduce $T\cdot
E_Z=T\cdot(bH_Z-H)=ab-1$ and $T_Z\cdot E=T_Z\cdot(aH-H_Z)=abz-z$. We remark that $iz=-K_Z\cdot L_Z=-K_W\cdot T_Z$, and
hence $iz=(r+1)H\cdot T_Z-(r-n-1)E\cdot T_Z=(r+1)bz-(r-n-1)(ab-1)z$
by Proposition \ref{prop:K}. This proves (i). Similarly,
$r+1=-K_{\p^r}\cdot L=-K_W\cdot T=iH_Z\cdot T-(r-m-1)E_Z\cdot
T=ia-(r-m-1)(ab-1)$, proving (ii). Finally, (iii) and (iv) easily
follows by eliminating $i$ in (i) and (ii).
\end{proof}

In particular, we obtain better bounds than those of \cite[Lemma
3.1]{e-sb}:

\begin{corollary}
Let $\Phi:\p^r\da Z$ be a special birational transformation of type
$(a,b)$. Then:
\begin{enumerate}
\item[(i)] $a\leq \min\{m+2, r/(r-n-1)\}$
\item[(ii)] $b\leq \min\{n+2+i-r-1, (i-1)/(r-m-1)\}$
\end{enumerate}
\end{corollary}

\begin{remark}\label{rem:fano}
We will use several times along the paper the classification of manifolds of small coindex $r+1-i\in\{0,1,2,3\}$. If $i=r+1$ then $Z=\p^r$, and if $i=r$ then $Z\subset\p^{r+1}$ is a quadric hypersurface by \cite{k-o}. If $i=r-1$ then $Z$ is a Del Pezzo manifold of degree $z\in\{3,4,5\}$ (see \cite{fuj}), and if $i=r-2$ then $Z$ is a Mukai manifold of degree $z\in\{4,6,8,10,12,14,16,18,22\}$ (see \cite{muk}).
 \end{remark}

Throughout the paper, and motivated by the results of \cite{c-k}, \cite{e-sb} and \cite{c-k2}, we will study the following extremal and next-to-extremal cases $n\in\{1,2,3\}$, $m\in\{0,1,2\}$, and $r=n+2$. We begin with a simple consequence of Proposition \ref{prop:num}.

\begin{theorem}\label{thm:m=0}
Let $\Phi:\p^r\da Z$ be a special birational transformation of type
$(a,b)$. If $m=0$ then $X\subset\p^{r-1}$ is a quadric hypersurface
and $Z\subset\p^{r+1}$ is a quadric hypersurface.
\end{theorem}

\begin{proof}
Since $m=0$, we get from Proposition \ref{prop:num}(iii) that $a=2$
and $r=n+2$. Therefore, Proposition \ref{prop:num}(iv) yields $b=(i-1)/(r-1)\leq r/(r-1)$. As $r\geq 3$, we deduce
$b=1$ and $i=r$. Finally, $\deg(\Sec_2(X))=1$ and hence $X\subset\p^{r-1}$ is a quadric
hypersurface by Proposition \ref{prop:sec}, so $Z\subset\p^{r+1}$ is a smooth quadric hypersurface (see Example \ref{ex:c.i} and Proposition \ref{prop:c.i}).
\end{proof}

Combined with Proposition \ref{prop:num}, the following result will be crucial in the sequel (cf. \cite[Formulae 0.3.]{c-k}):

\begin{proposition}[Fundamental formulae]\label{prop:ff}
Let $\Phi:\p^r\da Z$ be a special birational transformation of type
$(a,b)$, and let $d:=\deg(X)$. Then $d<a^{r-n}$ and
\begin{enumerate}
\item[(i)] $z=a^r-\binom{r}{n}a^nd-\sum^{n-1}_{i=0}\binom{r}{i}a^is_{n-i}$
\item[(ii)] $bz=a^{r-1}-\binom{r-1}{n-1}a^{n-1}d-\sum^{n-1}_{i=1}\binom{r-1}{i-1}a^{i-1}s_{n-i}$
\item[(iii)] $b^2z-e=a^{r-2}-\binom{r-2}{n-2}a^{n-2}d-\sum^{n-1}_{i=2}\binom{r-2}{i-2}a^{i-2}s_{n-i}$,
where $e:=Y\cdot\O_Z^{r-2}(1)$
\end{enumerate}
\end{proposition}

\begin{proof}
Since $X\subset\p^r$ is defined by polynomials of degree $a$, we deduce $d\leq a^{r-n}$. Furthermore,
if $d=a^{r-n}$ then $X\subset\p^r$ is the complete intersection of $r-n$ hypersurfaces of degree $a$, giving a contradiction. We recall that $H^i\cdot E^{r-i}=0$ for $n<i<r$ and
that $H^i\cdot E^{r-i}=(-1)^{r-i-1}s_{n-i}$ for $i\leq n$, where
$s_k:=\deg(s_k(N_{X/\p^r}))$ denotes the degree of the $k$-th Segre class of
the normal bundle $N_{X/\p^r}$ of $X$ in $\p^r$. Therefore, (i) follows from
the equality $z=H_Z^r=(aH-E)^r$ and (ii) follows from the equality
$bz=H_Z^{r-1}\cdot H=(aH-E)^{r-1}\cdot H$. Finally, $b^2z-e=H_Z^{r-2}\cdot H^2=(aH-E)^{r-2}\cdot H^2$ giving (iii).
\end{proof}

\begin{remark}\label{rem:e}
Note that $e=\deg(Y)$ if and only if $Y\subset Z$ has codimension $2$. Otherwise, $e=0$.
\end{remark}

\begin{remark}\label{rem:s}
The number $s_k$ can be obtained from the exact sequence $$0\to
TX\to T\p^r_{|X}\to N_{X/\p^r}\to 0$$ having in mind that $s(N_{X/\p^r})\cdot c(N_{X/\p^r})=1$.
Therefore, $$s(N_{X/\p^r})=c(T_X)\cdot
c(T\p^r_{|X})^{-1}=\sum_{i=0}^nc_i\cdot\sum_{i=0}^n(-1)^i\binom{r+i}{i}H_X^i$$
and $s_k:=s_k(N_{X/\p^r})\cdot H_X^{n-k}$, where $H_X$ denotes the hyperplane section of $X\subset\p^r$. In particular, $s_0=d$.
\end{remark}

\section{Case $n=1$}

\begin{proposition}[Numerology for $n=1$]\label{prop:num n=1}
Let $\Phi:\p^r\da Z$ be a special birational transformation of type $(a,b)$.
If $n=1$ then one of the following holds:
\begin{enumerate}
\item[(i)] $r=4$, $a=2$, $b\in\{1,2,3\}$, $i=b+2$ and $m=2$;
\item[(ii)] $r=3$, $a=3$, $b\in\{1,2,3\}$, $i=b+1$ and $m=1$;
\item[(iii)] $r=3$, $a=2$, $b=1$, $i=3$ and $m=0$.
\end{enumerate}
\end{proposition}

\begin{proof}
This is a simple consequence of Proposition \ref{prop:num}.
\end{proof}

Let $g:=g(X)$ denote the genus of $X$.

\begin{proposition}[Fundamental formulae for $n=1$]\label{prop:ff n=1}
Let $\Phi:\p^r\da Z$ be a special birational transformation of type
$(a,b)$. If $n=1$ then:
\begin{enumerate}
\item[(i)] $z=a^r+(r+1-ra)d+2g-2$
\item[(ii)] $bz=a^{r-1}-d$
\item[(iii)] $b^2z-e=a^{r-2}$
\end{enumerate}
\end{proposition}
\begin{proof}
They easily follow from Proposition \ref{prop:ff} and Remark \ref{rem:s}, since $s_0=d$ and $s_1=2-2g-(r+1)d$.
\end{proof}

%\begin{remark}
%The last formula does not appear in \cite{c-k}. Note that
%$b^2z-Y\cdot\O_Z(1)^{r-2}$ is the degree in $\p^r$ of the image by
%$\Psi$ of the general surface section of $Z$. In particular,
%$Y\cdot\O_Z(1)^{r-2}=0$ if and only if $m\leq 1$.
%\end{remark}

\begin{proposition}[Maximal list for $n=1$]\label{prop:n=1}
Let $\Phi:\p^r\da Z$ be a special birational transformation of type $(a,b)$.
If $n=1$ then one of the following holds:
\begin{enumerate}
\item[(i)] $r=4$, $a=2$, $b=3$, $d=5$, $g=1$ and $z=1$;
\item[(ii)] $r=4$, $a=2$, $b=2$, $d=4$, $g=0$ and $z=2$;
\item[(iii)] $r=4$, $a=2$, $b=1$, $d=4$, $g=1$ and $z=4$;
\item[(iv)] $r=4$, $a=2$, $b=1$, $d=3$, $g=0$ and $z=5$;
\item[(v)] $r=3$, $a=3$, $b=3$, $d=6$, $g=3$ and $z=1$;
\item[(vi)] $r=3$, $a=3$, $b=2$, $d=5$, $g=1$ and $z=2$;
\item[(vii)] $r=3$, $a=3$, $b=1$, $d=6$, $g=4$ and $z=3$;
\item[(viii)] $r=3$, $a=3$, $b=1$, $d=5$, $g=2$ and $z=4$;
\item[(ix)] $r=3$, $a=3$, $b=1$, $d=4$, $g=0$ and $z=5$;
\item[(x)] $r=3$, $a=2$, $b=1$, $d=2$, $g=0$ and $z=2$.
\end{enumerate}
\end{proposition}

\begin{proof}
We proceed according to Proposition \ref{prop:num n=1}. Assume $r=4$ and $a=2$. Then Proposition \ref{prop:ff n=1} yields $z=14-3d+2g$, $bz=8-d$ and $b^2z-e=4$. If $b=3$ then $i=5$ and hence $z=1$, so $d=5$ and $g=1$ giving (i). If $b=2$ then $i=4$ and hence $z=2$, so $d=4$ and $g=0$. This gives (ii). If $b=1$ then $i=3$ and hence $z\in\{3,4,5\}$ by the classification of Del Pezzo manifolds \cite{fuj}. If $z=3$ then $d=5$, $g=2$ and $e=-1$ so this case does not exist (see Remark \ref{rem:e}). If $z=4$ then $d=4$ and $g=1$, giving (iii). If $z=5$ then $d=3$ and $g=0$, giving (iv).

Assume now $r=3$ and $a=3$. Then we get $z=25-5d+2g$, $bz=9-d$ and $b^2z-e=3$ by Proposition \ref{prop:ff n=1}. If $b=3$ then $i=4$ and hence $z=1$, so $d=6$, $g=3$ and we get (v). If $b=2$ then $i=3$ and hence $z=2$, so $d=5$ and $g=1$, giving (vi). If $b=1$ then $i=2$ and hence $z\in\{3,4,5\}$. If $z=3$ then $d=6$ and $g=4$. This gives (vii). If $z=4$ then $d=5$ and $g=2$, giving (viii). If $z=5$ then $d=4$ and $g=0$, giving case (ix).

Finally, if $r=3$ and $a=2$ then $b=1$. So $d=2$ and $g=0$ by Proposition \ref{prop:sec}, and we get (x).
\end{proof}

Let us exclude first the cases in which $Z$ is singular in Proposition \ref{prop:n=1}.

\begin{lemma}\label{lem:n=1 singular}
In Proposition \ref{prop:n=1}, cases (iii) and (vii), $Z$ turns out to be singular. More precisely:
\begin{itemize}
\item In case (iii), $X\subset\p^3$ is a complete intersection of quadrics, c.i. $(2,2)$ for short, and $Z\subset\p^6$ is a singular c.i. $(2,2)$.
\item In case (vii), $X\subset\p^3$ is a c.i. $(2,3)$ and $Z\subset\p^4$ is a singular cubic hypersurface.
\end{itemize}
%Note that in both cases $\deg(\Sec_a(X))=a-1$ (cf.
%Proposition \ref{prop:sec}(i)) but there are two secant lines of $X$
%passing through a general point $p\in\Sec_a(X)$ (cf. Proposition
%\ref{prop:sec}(ii)). This shows that $\tau:W\to Z$ cannot be the
%blowing up of $Z$ along $Y$ (even on and open subset) and hence $Z$
%must be singular (cf. \cite[Theorem 1 and Proposition
%2.1.(b)]{e-sb}).
\end{lemma}

\begin{proof}
In both cases, the classification of $X\subset\p^r$ is well known. Concerning $Z$, case (iii) follows from a direct computation. For case (vii), see Example \ref{ex:c.i} and Proposition \ref{prop:c.i}.
\end{proof}

The main result of this section is the following:

\begin{theorem}\label{thm:n=1}
Let $\Phi:\p^r\da Z$ be a special birational transformation of type $(a,b)$.
If $n=1$ then one of the following holds:
\begin{enumerate}
\item[(I)] $r=4$, $(a,b)=(2,3)$, $X\subset\p^4$ is a quintic elliptic curve and $Z=\p^4$;
\item[(II)] $r=4$, $(a,b)=(2,2)$, $X\subset\p^4$ is a quartic rational curve and $Z\subset\p^5$ is a quadric hypersurface;
\item[(III)] $r=4$, $(a,b)=(2,1)$, $X\subset\p^4$ is a twisted cubic and $Z\subset\p^7$ is a linear section of $\g(1,4)\subset\p^9$;
\item[(IV)] $r=3$, $(a,b)=(3,3)$, $X\subset\p^3$ is a sextic curve of genus $3$ and $Z=\p^3$;
\item[(V)] $r=3$, $(a,b)=(3,2)$, $X\subset\p^3$ is a quintic elliptic curve and $Z\subset\p^4$ is a quadric hypersurface;
\item[(VI)] $r=3$, $(a,b)=(3,1)$, $X\subset\p^3$ is a quintic curve of genus $2$ and $Z\subset\p^5$ is a c.i. $(2,2)$;
\item[(VII)] $r=3$, $(a,b)=(3,1)$, $X\subset\p^3$ is a quartic rational curve and $Z\subset\p^6$ is a linear section of $\g(1,4)\subset\p^9$;
\item[(VIII)] $r=3$, $(a,b)=(2,1)$, $X\subset\p^2$ is a conic and $Z\subset\p^4$ is a quadric hypersurface.
\end{enumerate}
\end{theorem}

\begin{proof}
By Proposition \ref{prop:n=1} and Lemma \ref{lem:n=1 singular} we have to consider only eight cases. In all of them, the classification of $X\subset\p^r$ is well known. Concerning $Z$, cases (I) and (IV) are classical (see Remark \ref{rem:cremona} (i) and (ii), respectively). Case (II) is contained in Example \ref{ex:quadrics}(iv). Case (III) is obtained by a direct computation taking two hyperplane sections in Example \ref{ex:quadrics}(i). Case (V) is given either in Example \ref{ex:g(1,r)}, or in Example \ref{ex:hypersurfaces} (see Remark \ref{rem:hypersurfaces}). Cases (VI) and (VIII) are given in Example \ref{ex:c.i} and Proposition \ref{prop:c.i}. Finally, case (VII) is given either in Example \ref{ex:g(1,r+1)}, or in Example \ref{ex:degree21}.
\end{proof}

In particular, we recover the following result (see \cite[Theorem 2.2]{c-k}):

\begin{corollary}[Crauder-Katz]
Let $\Phi:\p^r\da\p^r$ be a special Cremona transformation of type
$(a,b)$. If $n=1$ then either $r=4$, $(a,b)=(2,3)$ and $X\subset\p^4$
is a quintic elliptic curve, or else $r=3$, $(a,b)=(3,3)$ and
$X\subset\p^3$ is a sextic curve of genus $3$.
\end{corollary}

Theorem \ref{thm:n=1} also yields the following consequence for $m=1$ (cf. Theorem \ref{thm:m=0}):

\begin{theorem}\label{thm:m=1}
Let $\Phi:\p^r\da Z$ be a special birational transformation of type
$(a,b)$. If $m=1$ then $a=3$, $r=n+2$ and either $X\subset\p^r$ and $Z$ are as in Theorem \ref{thm:n=1}, cases (IV)-(VII), or $X\subset\p^4$ is a quintic elliptic scroll and $Z=\p^4$, or $X\subset\p^r$ is a quintic Castelnuovo manifold, $r\in\{3,4,5\}$, and $Z\subset\p^{r+2}$ is a c.i. $(2,2)$.
\end{theorem}

\begin{proof}
If $m=1$ we deduce from Proposition \ref{prop:num}(iii) that $a=3$ and $r=n+2$. Therefore, Proposition \ref{prop:num}(iv) yields $b=(i-1)/(r-2)\leq r/(r-2)$. So $b\in\{1,2,3\}$ for $r=3$, $b\in\{1,2\}$ for $r=4$, and $b=1$ for $r\geq 5$. If $r=3$ then $n=1$ and we conclude by Theorem \ref{thm:n=1}. If $r=4$ and $b=2$ then $X\subset\p^4$ is a quintic elliptic scroll by \cite{aure}, since $\Sec_3(X)\neq\p^4$ and $\deg(\Sec_3(X))=5$, and $Z=\p^4$. Finally, if $r\geq 4$ and $b=1$ then $\deg(\Sec_3(X))=2$ by Proposition \ref{prop:sec}(i). Hence $X\subset\p^r$ is a divisor on a quadric hypersurface, so either it is a c.i. $(2,3)$ and hence $Z\subset\p^{r+1}$ would be a singular cubic hypersurface by Example \ref{ex:c.i} and Proposition \ref{prop:c.i}, or else it is linked to a linear $\p^n$ by a c.i. $(2,3)$ (i.e. it is a quintic Castelnuovo manifold). In that case, $Z\subset\p^{r+2}$ is a c.i. $(2,2)$ by Example \ref{ex:c.i} and Proposition \ref{prop:c.i}.
\end{proof}

\begin{remark}
For $r=3$, a similar result to Theorem \ref{thm:n=1} has been independently obtained in \cite[Theorem 7]{pan}.
\end{remark}

\section{Case $n=2$}\label{section:n=2}

\begin{proposition}[Numerology for $n=2$]\label{prop:num n=2}
Let $\Phi:\p^r\da Z$ be a special birational transformation of type
$(a,b)$. If $n=2$ then one of the following holds:
\begin{enumerate}
\item[(i)] $r=6$, $a=2$, $b\in\{1,2,3,4\}$, $i=b+3$ and $m=4$;
\item[(ii)] $r=5$, $a=2$, $b\in\{1,2\}$, $i=2b+2$ and $m=2$;
\item[(iii)] $r=4$, $a=4$, $b\in\{1,2,3,4\}$, $i=b+1$ and $m=2$;
\item[(iv)] $r=4$, $a=3$, $b\in\{1,2\}$, $i=2b+1$ and $m=1$;
\item[(v)] $r=4$, $a=2$, $b=1$, $i=4$ and $m=0$.
\end{enumerate}
\end{proposition}

\begin{proof}
It easily follows from Proposition \ref{prop:num}.
\end{proof}

Let $g$ denote the sectional genus of $X\subset\p^r$.

\begin{proposition}[Fundamental formulae for $n=2$]\label{prop:ff n=2}
Let $\Phi:\p^r\da Z$ be a special birational transformation of type
$(a,b)$. If $n=2$ then:
\begin{enumerate}
\item[(i)] $s_1=2-2g-rd$
\item[(ii)] $s_2=c_2-(r+1)(2-2g)+\binom{r+1}{2}d$
\end{enumerate}
\end{proposition}

\begin{proof} We deduce from Remark \ref{rem:s} that
$s(N_{X/\p^r})=(1+c_1+c_2)(1-(r+1)H_X+\binom{r+2}{2}H_X^2)$, whence $s_1=c_1\cdot
H_X-(r+1)H_X^2$ and $s_2=c_2-(r+1)c_1\cdot H_X+\binom{r+2}{2}H_X^2$. Note that $2g-2=H_X^2+H_X\cdot K_X$ by the adjunction formula, so we get
(i) and (ii).
\end{proof}

\textbf{Multisecant formulae for surfaces.} Let us recall some formulae, that we denote by $N_{k,r}$, concerning
the $k$-secant lines to a smooth surface
$X\subset\p^r$. Let $K^2:=K_X^2$ and $\chi:=\chi(\O_X)$. We start with the well-known double point formula, that gives the number of $2$-secant lines of $X\subset\p^r$ meeting a general $\p^{r-5}$ (so, in particular, $N_{2,4}=0$):
\begin{equation*}
    N_{2,r}=\frac{1}{2}(d^2-5d-10(g-1)-2K^2+12\chi)
\end{equation*}

Now we include several formulae due to Le Barz concerning higher $k$-secant lines (see \cite{le-barz}). Let
$X\subset\p^6$ be a smooth surface. The number of $3$-secant lines
meeting a general $4$-dimensional linear subspace is given by
\begin{equation*}
    N_{3,6}=\frac{1}{3}(d^3-12d^2+d(18\chi-18g-3K^2+50)-12(10\chi-11g-2K^2+11))
\end{equation*}

\begin{remark}\label{rem:scroll}
This formula does not work if $X\subset\p^6$ is a scroll over a curve.
\end{remark}

Let $X\subset\p^4$ be a smooth surface. The number of $3$-secant
lines passing a general point is given by
\begin{equation*}
    N_{3,4}=\frac{1}{6}(12\chi+d^3-6d^2+d(11-6g)+18(g-1))
\end{equation*}
If $X\subset\p^4$ is non-degenerate, we remark that $N_{3,4}=0$ if and only if $X\subset\p^4$ is either a
quintic elliptic scroll or it is contained in a quadric hypersurface (see \cite{aure}). The number of $4$-secant lines of $X\subset\p^4$ meeting a general line is given by
\begin{equation*}
    N_{4,4}=\frac{1}{8}(8\chi(2d-9)+d^4-10d^3+d^2(35-8g)+2d(28g-33)+4(g^2-25g+24))
\end{equation*}
and the number of $5$-secant lines of $X\subset\p^4$ meeting
a general plane is
\smallskip
\begin{center}
$N_{5,4}=\frac{1}{24}(24\chi(d^2-12d-2(g-22))+d^5-16d^4+d^3(95-12g)+52d^2(3g-5)+12d(2g^2-62g+43)-120(g^2-11g+10))$
\end{center}
\smallskip
Finally, the number of $6$-secant lines of $X\subset\p^4$ is given by
\smallskip
\begin{center}
$N_{6,4}=\frac{1}{144}(288\chi^2+24\chi(2d^3-39d^2+d(355-12g)+90(g-16))+d^6-21d^5+d^4(157-18g)+3d^3(116g-65)+2d^2(36g^2-1431g-1663)-12d(72g^2-1033g-774)-24(g^3-111g^2+854g -744))-\sum_j\binom{7+l_j}{6}$
\end{center}
\smallskip
where $l_j$ denotes the self-intersection of every line $L_j\subset X$.

\begin{remark}\label{rem:scroll P4}
The formulae $N_{5,4}$ and $N_{6,4}$ do not work if $X\subset\p^4$ is a scroll over a curve. There are just two non-degenerate scrolls in $\p^4$, the rational normal scroll and the quintic elliptic scroll (see \cite{lanteri}). The only degenerate scroll in $\p^4$ is $\p^1\times\p^1\subset\p^3$.
\end{remark}

\begin{proposition}[Maximal list for $n=2$]\label{prop:n=2}
Let $\Phi:\p^r\da Z$ be a special birational transformation of type
$(a,b)$. If $n=2$ then one of the following holds:
\begin{enumerate}
\item[(i)] $r=6$, $a=2$, $b=4$, $X$ is a scroll, $d=7$, $g=1$ and $z=1$;
\item[(ii)] $r=6$, $a=2$, $b=2$, $X$ is a scroll, $d=5$, $g=0$ and $z=5$;
\item[(iii)] $r=6$, $a=2$, $b=1$, $X$ is a scroll, $d=4$, $g=0$ and $z=14$;
\item[(iv)] $r=6$, $a=2$, $b=4$, $d=8$, $g=3$, $K^2=1$, $\chi=1$ and $z=1$;
\item[(v)] $r=6$, $a=2$, $b=3$, $d=7$, $g=2$, $K^2=3$, $\chi=1$ and $z=2$;
\item[(vi)] $r=6$, $a=2$, $b=2$, $d=6$, $g=1$, $K^2=6$, $\chi=1$ and $z=4$;
\item[(vii)] $r=6$, $a=2$, $b=1$, $d=5$, $g=1$, $K^2=5$, $\chi=1$ and $z=12$;
\item[(viii)] $r=5$, $a=2$, $b=2$, $X\subset\p^5$ is a Veronese surface and $z=1$;
\item[(ix)] $r=5$, $a=2$, $b=1$, $X\subset\p^4$ is a rational normal scroll and $z=5$;
\item[(x)] $r=4$, $a=4$, $b=4$, $d=10$, $g=11$, $K^2=5$, $\chi=5$ and $z=1$;
\item[(xi)] $r=4$, $a=4$, $b=2$, $d=9$, $g=8$, $K^2=-5$, $\chi=2$ and $z=3$;
\item[(xii)] $r=4$, $a=4$, $b=2$, $d=8$, $g=6$, $K^2=-1$, $\chi=2$ and $z=5$;
\item[(xiii)] $r=4$, $a=4$, $b=1$, $d=12$, $g=19$, $K^2=48$, $\chi=16$ and $z=4$;
\item[(xiv)] $r=4$, $a=4$, $b=1$, $d=10$, $g=12$, $K^2=12$, $\chi=7$ and $z=6$;
\item[(xv)] $r=4$, $a=4$, $b=1$, $d=9$, $g=9$, $K^2=2$, $\chi=4$ and $z=8$;
\item[(xvi)] $r=4$, $a=4$, $b=1$, $d=7$, $g=4$, $K^2=-2$, $\chi=1$ and $z=14$;
%\item[(xiv)] $r=4$, $a=4$, $b=1$, $d=9$, $g=7$ and $Z\subset\p^5$ is
%a quartic hypersurface \textbf{DOES NOT EXIST}
\item[(xvii)] $r=4$, $a=3$, $b=2$, $d=5$, $g=1$, $K^2=0$, $\chi=0$ and $z=1$;
\item[(xviii)] $r=4$, $a=3$, $b=1$, $d=6$, $g=4$, $K^2=0$, $\chi=2$ and $z=3$;
\item[(xix)] $r=4$, $a=3$, $b=1$, $d=5$, $g=2$, $K^2=1$, $\chi=1$ and $z=4$;
\item[(xx)] $r=4$, $a=2$, $b=1$, $X\subset\p^3$ is a quadric and $z=2$.
\end{enumerate}
\end{proposition}

\begin{proof}
We proceed according to Proposition \ref{prop:num n=2}. Assume $r=6$ and $a=2$. If $X\subset\p^6$ is a scroll then $K^2=8(1-g)$ and
$\chi=1-g$. So we get $z=50-9d+14g$, $bz=30-4d+2g$ and $b^2z-e=16-d$ by Propositions \ref{prop:ff} and \ref{prop:ff n=2}, whence $z=(160-19d)/(7b-1)$. If $b=4$ then $i=7$ and hence $z=1$, so $d=7$ and $g=1$, giving (i). If $b=3$  then $i=6$ and hence $z=2$, but $d\notin\n$. If $b=2$ then $i=5$ and hence $z\in\{3,4,5\}$ by the classification of Del Pezzo manifolds \cite{fuj}, but the only admissible case is $z=5$, $d=5$ and $g=0$, giving (ii). If $b=1$ then $i=4$ and hence $z\in\{4,6,8,10,12,14,16\}$ by the classification of Mukai manifolds \cite{muk}, but the only admissible case is $z=14$, $d=4$ and $g=0$. This gives (iii).

If $X\subset\p^6$ is not a scroll we deduce from Propositions
\ref{prop:ff} and \ref{prop:ff n=2} the fundamental formulae
$z=54-9d+10g+K^2-12\chi$, $bz=30-4d+2g$ and $b^2z-e=16-d$. Furthermore,
$N_{2,6}=\deg(\Sec_2(X))=2b-1$ and hence
$d^2-5d-5(2g-2)-2K^2+12\chi=4b-2$. In addition $N_{3,6}=0$, as
$X\subset\p^6$ is defined by quadrics, so
$d^3-12d^2+d(18\chi-18g-3K^2+50)-12(10\chi-11g-2K^2+11)=0$. Hence
we can eliminate $g=(bz+4d-30)/2$, $K^2=-z-4b+d^2-14d+66$ and
$\chi=(z(5b-2)-4b+d^2-3d-30)/12$, so we get
$$(b(3d-32)+8)z=4b(3d-28)-d^3+7d^2+70d-456$$ If $b=4$ then $i=7$ and hence $z=1$, so we deduce $d=8$ and hence $g=3$, $K^2=1$ and $\chi=1$, giving (iv). If
$b=3$ then $i=6$ and hence $z=2$, so we deduce $d=7$ and hence $g=2$, $K^2=3$ and $\chi=1$. This gives (v). If $b=2$ then $i=5$ and hence $z\in\{3,4,5\}$. If $z\in\{3,5\}$ then $d\notin\n$. If $z=4$ then $d=6$, whence $g=1$, $K^2=6$, $\chi=1$ and we get (vi). If $b=1$ then $i=4$ and hence $z\in\{4,6,8,10,12,14,16\}$. If $z\in\{4,6,8,10,14\}$ then $d\notin\n$. If $z=12$ then $d=5$. In this case, $g=1$, $K^2=5$, $\chi=1$ and we get (vii). If $z=16$ then $d=4$, but in this case $\chi\notin\z$.

Consider now $r=5$ and $a=2$. In this case, we can argue geometrically instead of computing the numerical invariants. As $\dim(\Sec_2(X))=4$, we deduce from Severi's theorem \cite{severi} that either $X\subset\p^5$ is a Veronese surface, giving (viii), or else $X\subset\p^4$. Furthermore, $X\subset\p^4$ is swept out by a
$2$-dimensional family of conics by Proposition \ref{prop:sec}(ii),
whence $X\subset\p^4$ is a rational normal scroll and we get (ix).

Assume $r=4$. If $a=4$ then we get from Propositions
\ref{prop:ff} and \ref{prop:ff n=2} the fundamental formulae
$z=234-42d+22g+K^2-12\chi$, $bz=62-8d+2g$ and $b^2z-e=16-d$. Moreover, $N_{2,4}=0$
and hence $d^2-5d-10(g-1)-2K^2+12\chi=0$. Then we obtain
$g=(bz+8d-62)/2$, $K^2=z(6b-1)+d^2+d-128$ and
$\chi=(z(17b-2)+d^2+47d-576)/12$. Note that $X\subset\p^4$ cannot be a scroll by Remark \ref{rem:scroll P4}, since they do not have any $4$-secant line. So $N_{5,4}=0$ as $X\subset\p^4$ is defined by quartics, and hence we get
\begin{center}
$2bz^2(b(3d-32)+2)-2z(b(3d^3-103d^2+1185d-4568)+2(d^2-20d+106))+d^5-62d^4+1529d^3-18764d^2+114688d-279552=0$
\end{center}
after substituting $g$, $K^2$ and $\chi$. We can assume $d<16$ by Proposition \ref{prop:ff}. If $b=4$ then $i=5$ and hence $z=1$, so we deduce $d=10$. In this case, $g=11$, $K^2=5$ and $\chi=5$.
This gives (x). If $b=3$ then $i=4$ and hence $z=2$, so we deduce $d=10$ but $\chi\notin\z$. If $b=2$ then $i=3$ and hence $z\in\{3,4,5\}$. If $z=3$ then either $d=9$, whence $g=8$, $K^2=-5$ and $\chi=2$, giving (xi), or else $d=10$, and hence $\chi\notin\z$. If $z=4$ then $d\notin\n$ . If $z=5$ then $d=8$, whence $g=6$, $K^2=-1$ and $\chi=2$, giving (xii). Finally, if $b=1$ then $i=2$ and hence $z\in\{4,6,8,10,12,14,16,18\}$. If $z=4$ then $d=12$, whence $g=19$, $K^2=48$ and $\chi=16$, giving (xiii). If $z=6$ then $d=10$, whence $g=12$, $K^2=12$ and $\chi=7$. This gives (xiv). If $z=8$ then $d=9$, whence $g=9$, $K^2=2$ and $\chi=4$, giving (xv). If $z=10$ then $d=8$ but $\chi\notin\z$. If $z=14$ then $d=7$, whence $g=4$, $K^2=-2$ and $\chi=1$, giving (xvi). If $z\in\{12,16,18\}$ then $d\notin\n$.

Now we consider $a=3$. Then we get from Propositions \ref{prop:ff} and \ref{prop:ff n=2} the fundamental formulae
$z=67-16d+14g+K^2-12\chi$, $bz=25-5d+2g$ and $b^2z-e=9-d$. Since
$N_{2,4}=0$, we deduce $g=(bz+5d-25)/2$, $K^2=z(2b-1)+d^2-11d+27$
and $\chi=(z(9b-2)+d^2+8d-81)/12$. As
$X\subset\p^4$ is defined by cubics, $N_{4,4}=0$ and we deduce
$$3b^2z^2-2z(3b(2d^2-25d+77)+2(2d-9))+3d^4-86d^3+914d^2-4266d+7371=0$$
If $b=2$ then $i=5$ and hence $z=1$, so $d=5$. In this case, $g=1$, $K^2=0$ and $\chi=0$, giving (xvii). On the other hand, if $b=1$ then $i=3$ and hence $z\in\{3,4,5\}$. If $z=3$ then $d=6$ and hence $g=4$, $K^2=0$ and $\chi=2$, giving (xviii). If $z=4$ then $d=5$, whence $g=2$, $K^2=1$ and $\chi=1$, giving (xix). If $z=5$ then $d\notin\n$.

Finally, if $a=2$ then $b=1$ so $X\subset\p^3$ is a quadric by Proposition \ref{prop:sec} and we get (xx).
\end{proof}

Let us study first the cases where $Z$ is singular in Proposition \ref{prop:n=2}.

\begin{lemma}\label{lem:n=2 singular}
In Proposition \ref{prop:n=2}, $Z$ turns out to be singular in cases (iii), (v), (vi), (vii), (xiii), (xiv) and (xviii). More precisely:
\begin{itemize}
\item In case (iii), $X\subset\p^5$ is a rational normal scroll and $Z\subset\p^{12}$ is a singular linear section of $\g(1,5)\subset\p^{14}$.
\item In case (v), $X\subset\p^6$ is the blowing-up of $\p^2$ along six points embedded by $|4L-2E_0-\sum_{i=1}^5 E_i|$ and $Z\subset\p^7$ is a quadric cone of rank at most $6$.
\item In case (vi), $X\subset\p^6$ is a Del Pezzo surface and $Z\subset\p^8$ is a singular c.i. $(2,2)$.
\item In case (vii), $X\subset\p^5$ is a Del Pezzo surface and $Z\subset\p^{11}$ is a singular linear section of the $10$-dimensional spinor variety $S_4\subset\p^{15}$.
\item In case (xiii), $X\subset\p^4$ is a c.i. $(3,4)$ and $Z\subset\p^5$ is a singular quartic hypersurface.
\item In case (xiv), $X\subset\p^4$ is linked to a quadric in a c.i. $(3,4)$ and $Z\subset\p^6$ is a singular c.i. $(2,3)$.
\item In case (xviii), $X\subset\p^4$ is a c.i. $(2,3)$ and $Z\subset\p^5$ is a singular cubic hypersurface.
\end{itemize}
\end{lemma}

\begin{proof}
The classification of $X\subset\p^r$ is well known in cases (iii), (vi), (vii) and (xviii). In case (v), it was classified in \cite{ion}. In case (xiii), it is a c.i. by \cite{g-p}. Finally, in case (xiv) the classification follows from \cite[Theorem 0.1]{th-ranestad}. Now we focus on $Z$. In cases (iii), (vi) and (vii) the description of $Z$ follows from a direct computation taking a linear section in Example \ref{ex:quadrics} (i), (iv) and (ii), respectively. In cases (xiii), (xiv) and (xviii) it follows from Example \ref{ex:c.i} and Proposition \ref{prop:c.i}. Finally, in case (v) we remark that $X\subset\p^6$ contains a $3$-dimensional family of quintic elliptic curves corresponding to the curves in the linear system $|3L-\sum_{i=0}^5 E_i|$. The restriction of $\Phi$ to a $\p^4$ containing a quintic elliptic curve gives a Cremona transformation of type $(2,3)$. So $Z$ is swept out by a $3$-dimensional family of $\p^4$'s, and hence we deduce that it is a quadric cone of rank at most $6$.
\end{proof}

The main result of this section is the following:

\begin{theorem}\label{thm:n=2}
Let $\Phi:\p^r\da Z$ be a special birational transformation of type
$(a,b)$. If $n=2$ then one of the following holds:
\begin{enumerate}
\item[(I)] $r=6$, $(a,b)=(2,4)$, $X\subset\p^6$ is a septic elliptic scroll of invariant $-1$ and $Z=\p^6$;
\item[(II)] $r=6$, $(a,b)=(2,2)$, $X\subset\p^6$ is a rational normal scroll and $Z=\g(1,4)\subset\p^9$;
\item[(III)] $r=6$, $(a,b)=(2,4)$, $X\subset\p^6$ is the blowing-up of $\p^2$ along eight points embedded by $|4L-\sum_{i=1}^8 E_i|$ and $Z=\p^6$;
\item[(IV)] $r=5$, $(a,b)=(2,2)$, $X\subset\p^5$ is a Veronese surface and $Z=\p^5$;
\item[(V)] $r=5$, $(a,b)=(2,1)$, $X\subset\p^4$ is a rational normal scroll and $Z\subset\p^8$ is a hyperplane section of $\g(1,4)\subset\p^9$;
\item[(VI)] $r=4$, $(a,b)=(4,4)$, $X\subset\p^4$ is linked to a Bordiga surface in a c.i. $(4,4)$ and $Z=\p^4$;
\item[(VII)] $r=4$, $(a,b)=(4,2)$, $X\subset\p^4$ is a non-minimal K3 surface with five $(-1)$-lines and $Z\subset\p^5$ is a cubic hypersurface;
\item[(VIII)] $r=4$, $(a,b)=(4,2)$, $X\subset\p^4$ is a non-minimal K3 surface with a $(-1)$-line and $Z\subset\p^7$ is a linear section of $\g(1,4)\subset\p^9$;
\item[(IX)] $r=4$, $(a,b)=(4,1)$, $X\subset\p^4$ is linked to a cubic scroll in a c.i. $(3,4)$ and $Z\subset\p^7$ is a c.i. $(2,2,2)$;
\item[(X)] $r=4$, $(a,b)=(4,1)$, $X\subset\p^4$ is the blowing-up of $\p^2$ along eleven points embedded by $|6L-\sum_{i=1}^5 E_i-\sum_{i=6}^{11} 2E_i|$ and $Z\subset\p^{10}$ is a linear section of $\g(1,5)\subset\p^{14}$;
\item[(XI)] $r=4$, $(a,b)=(3,2)$, $X\subset\p^4$ is a quintic elliptic scroll and $Z=\p^4$;
\item[(XII)] $r=4$, $(a,b)=(3,1)$, $X\subset\p^4$ is a quintic Castelnuovo surface and $Z\subset\p^6$ is a c.i. $(2,2)$;
\item[(XIII)] $r=4$, $(a,b)=(2,1)$, $X\subset\p^3$ is a quadric and $Z\subset\p^5$ is a quadric hypersurface.
\end{enumerate}
\end{theorem}

\begin{proof}
We proceed according to Proposition \ref{prop:n=2} and Lemma \ref{lem:n=2 singular}. The classification of $X\subset\p^r$ follows from the classification of surfaces of small degree: In case (i), we get (I) (see \cite[p. 304]{c-k} to compute the invariant of the scroll). In case (ii), we obtain (II). In case (iv), we get (III) by \cite{ion3}. Cases (viii) and (ix) yield cases (IV) and (V), respectively. In case (x), we get (VI) by \cite[Theorem 0.1]{th-ranestad}. In case (xi), we obtain (VII) by \cite[Theorem 0.1]{a-r}. Case (xii) yields (VIII) by \cite[Theorem 0.1]{oko}. In case (xv), we get (IX) by \cite[Theorem 0.1]{a-r} and case (xvi) gives (X) by \cite{ion}. Finally, cases (xvii), (xix) and (xx) yield (XI)-(XIII), respectively.

The description of $Z$ is well known in cases (I), (II), (III), (IV), (VI) and (XI). See Remark \ref{rem:cremona}(i), Example \ref{ex:quadrics}(iii), \cite{s-t3} or Remark \ref{rem:cremona}(iii) with $k=3$ and $d=2$, Example \ref{ex:quadrics}(iv), Remark \ref{rem:cremona}(ii) and the paragraph before Example \ref{ex:hypersurfaces}, respectively. Case (V) follows from a direct computation taking a hyperplane section in Example \ref{ex:quadrics}(i). For cases (IX), (XII) and (XIII), see Example \ref{ex:c.i} and Proposition \ref{prop:c.i}. For cases (X) and (VIII) see Examples \ref{ex:g(1,r+1)} and \ref{ex:g(1,r)}, respectively. Finally, for case (VII) see Example \ref{ex:hypersurfaces} and Remark \ref{rem:hypersurfaces} (note that ${\mathcal I}_X$ is given by a resolution $0\to T_{\p^4}(-6)\oplus\O_{\p^4}(-5)\to\O_{\p^4}(-4)^{\oplus6}\to {\mathcal I}_X\to 0$ by \cite[Theorem 0.3.c)]{a-r}).
\end{proof}

In particular, we recover the following result (see \cite[Theorem 3.3]{c-k}):

\begin{corollary}[Crauder-Katz]
Let $\Phi:\p^r\da\p^r$ be a special Cremona transformation of type
$(a,b)$. If $n=2$ then either $r=6$, $(a,b)=(2,4)$ and $X\subset\p^6$
is either an octic rational surface or a septic elliptic scroll of invariant $-1$, or
$r=5$, $(a,b)=(2,2)$ and $X\subset\p^5$ is a Veronese surface, or $r=4$,
$(a,b)=(3,2)$ and $X\subset\p^4$ is a quintic elliptic scroll, or
$r=4$, $(a,b)=(4,4)$ and $X\subset\p^4$ is a determinantal surface given by the $4\times4$-minors of a
$4\times5$-matrix of linear forms.
\end{corollary}

\section{Case $n=3$}

\begin{proposition}[Numerology for $n=3$]\label{prop:num n=3}
Let $\Phi:\p^r\da Z$ be a special birational transformation of type
$(a,b)$. If $n=3$ then one of the following holds:
\begin{enumerate}
\item[(i)] $r=8$, $a=2$, $b\in\{1,2,3,4,5\}$, $i=b+4$ and $m=6$;
\item[(ii)] $r=7$, $a=2$, $b\in\{1,2\}$, $i=2b+3$ and $m=4$;
\item[(iii)] $r=6$, $a=3$, $b=\{1,2,3,4,5\}$, $i=b+2$ and $m=4$;
\item[(iv)] $r=6$, $a=2$, $b=1$, $i=5$ and $m=2$;
\item[(v)] $r=5$, $a=5$, $b\in\{1,2,3,4,5\}$, $i=b+1$ and $m=3$;
\item[(vi)] $r=5$, $a=4$, $b\in\{1,2\}$, $i=2b+1$ and $m=2$;
\item[(vii)] $r=5$, $a=3$, $b=1$, $i=4$ and $m=1$;
\item[(viii)] $r=5$, $a=2$, $b=1$, $i=5$ and $m=0$.
\end{enumerate}
\end{proposition}

\begin{proof}
It easily follows from Proposition \ref{prop:num}.
\end{proof}

\begin{remark}
Cases (ii), (iv), (vii) and (viii) are easy to classify by using geometric arguments (see Theorem \ref{thm:n=3 easy} below). For cases (v) and (vi), the situation is more complicated. They will be treated in Theorem \ref{thm:n=3 r=5}. On the other hand, for cases (i), (iii) the problem seems to be much more difficult. As far as we know, even for special Cremona transformations cases (i) and (iii) remain open (see \cite[Corollary 1]{c-k2}), and
case (v) was obtained in \cite[Theorem 3.2]{e-sb}. We will extend this result in Corollary \ref{cor:r=n+2}.
\end{remark}

\begin{theorem}[Easy cases]\label{thm:n=3 easy}
Let $\Phi:\p^r\da Z$ be a special birational transformation of type
$(a,b)$, and let $n=3$.
\begin{enumerate}
\item[(I)] If $r=7$ then $a=2$ and either
\begin{itemize}
\item $b=2$, $X\subset\p^7$ is a hyperplane section of $\p^2\times\p^2\subset\p^8$ and $Z\subset\p^8$ is a quadric hypersurface, or 
\item $b=1$, $X\subset\p^6$ is a rational normal scroll and $Z\subset\p^{13}$ is a hyperplane section of $\g(1,5)\subset\p^{14}$, or
\item $b=1$, $X\subset\p^6$ is a linear section of $\g(1,4)$ and $Z\subset\p^{12}$ is a linear section of the $10$-dimensional spinor variety $S_4\subset\p^{15}$.
\end{itemize}
\item[(II)] If $r=6$ and $a=2$ then $b=1$, $X\subset\p^5$ is the Segre embedding of $\p^1\times\p^2$ and $Z=\g(1,4)\subset\p^9$.
\item[(III)] If $r=5$ and $a=3$ then $b=1$, $X\subset\p^5$ is a quintic Castelnuovo threefold and $Z\subset\p^7$ is a c.i. $(2,2)$.
\item[(IV)] If $r=5$ and $a=2$ then $b=1$, $X\subset\p^4$ is a quadric hypersurface and $Z\subset\p^6$ is a quadric hypersurface.
\end{enumerate}
\end{theorem}

\begin{proof}
If $r=7$ then $a=2$ and $b\in\{1,2\}$ by Proposition \ref{prop:num n=3}(ii). If $b=2$ then $\deg(\Sec_2(X))=3$ and $X\subset\p^7$ is a secant defective threefold by Proposition \ref{prop:sec}. Hence $X\subset\p^7$ is a hyperplane section of $\p^2\times\p^2\subset\p^8$ by \cite{fuj2}, and $Z\subset\p^8$ is a quadric hypersurface (see Example \ref{ex:quadrics}(iv)). On the other hand, if $b=1$ then $\deg(\Sec_2(X))=1$ by Proposition \ref{prop:sec}(i). In particular, $X\subset\p^6$ and $d<8$, so we easily deduce from \cite{ion} that $X\subset\p^6$ is either a rational normal scroll, or a linear section of $\g(1,4)\subset\p^9$. In the first case $Z\subset\p^{13}$ is a hyperplane section of $\g(1,5)\subset\p^{14}$, and in the second case $Z\subset\p^{12}$ is a linear section of the $10$-dimensional spinor variety $S_4\subset\p^{15}$. This follows by a direct computation taking a hyperplane section in Example \ref{ex:quadrics}(i) and three hyperplane sections in Example \ref{ex:quadrics}(ii), respectively, giving (I). If $r=6$ and $a=2$, then $b=1$ by Proposition \ref{prop:num n=3}(iv) and hence $\deg(\Sec_2(X))=1$ by Proposition \ref{prop:sec}(i). Therefore $X\subset\p^5$ and $d<4$, so it is the Segre embedding of $\p^1\times\p^2$ and $Z=\g(1,4)\subset\p^9$ (see Example \ref{ex:quadrics}(i)), giving (II). Finally, (III) and (IV) follow from Proposition \ref{prop:num n=3}(vii) and (viii), and Theorems \ref{thm:m=1} and \ref{thm:m=0}, respectively.
\end{proof}

In this section, $S\subset\p^{r-1}$ denotes a general hyperplane section of
$X\subset\p^r$ and $C\subset\p^{r-2}$ denotes a general hyperplane section of
$S\subset\p^{r-1}$. The genus of $C$ is denoted by $g$.

\begin{proposition}[Fundamental formulae for $n=3$]\label{prop:ff n=3}
Let $\Phi:\p^r\da Z$ be a special birational transformation of type
$(a,b)$. If $n=3$ then:
\begin{enumerate}
\item[(i)] $s_1=2-2g-(r-1)d$
\item[(ii)] $s_2=c_2(S)-r(2-2g)+\binom{r}{2}d$
\item[(iii)] $s_3=c_3(X)-(r+1)c_2(S)+\binom{r+1}{2}(2-2g)-\binom{r+1}{3}d$
%\item[(i)] $z=a^r-\binom{r}{3}a^3d-\binom{r}{2}a^2s_1-ras_2-s_3$
%\item[(ii)] $bz=a^{r-1}-\binom{r-1}{2}a^2d-(r-1)as_1-s_2$
%\item[(iii)] $b^2z-\deg(Y)=a^{r-2}-(r-2)ad-s_1$
\end{enumerate}
\end{proposition}

\begin{proof}
We deduce from Remark \ref{rem:s} that
$s(N_{X/\p^r})=(1+c_1+c_2+c_3)(1-(r+1)H_X+\binom{r+2}{2}H_X^2-\binom{r+3}{3}H_X^3)$,
so $s_1=c_1\cdot H_X^2-(r+1)H_X^3$, $s_2=c_2\cdot H_X-(r+1)c_1\cdot H_X^2+\binom{r+2}{2}H_X^3$
and $s_3=c_3-(r+1)c_2\cdot H_X+\binom{r+2}{2}c_1\cdot H_X^2-\binom{r+3}{3}H_X^3$. We
recall that $K_S=K_X\cdot H_X+H_X^2$ and $\deg(K_C)=K_X\cdot H_X^2+2H_X^3=2g-2$ by the adjunction
formula, and that $c_2(S)=c_2\cdot H_X+K_S\cdot H_S$ by the exact sequence
$0\to T_S\to {T_X}_{|S}\to \O_S(1)\to 0$, so $c_2\cdot H_X=c_2(S)+2-2g+d$. Then (i), (ii) and (iii) follow
from a simple computation.
\end{proof}

Furthermore, if $r=5$ we have two more well-known formulae:

\begin{proposition}\label{prop:f in P5}
If $r=5$ then we get:
\begin{enumerate}
\item[(i)] $K^3_X=-5d^2+d(2g+25)+24(g-1)-36\chi(\O_S)-24\chi(\O_X)$
\item[(ii)] $2K_S^2=d^2-5d-10(g-1)+12\chi(\O_S)$
\end{enumerate}
\end{proposition}

We now consider case (vi) in Proposition \ref{prop:num n=3}:

\begin{proposition}[Maximal list for $n=3$ and $(a,b)=(4,b)$]\label{prop:(4,b)}
Let $\Phi:\p^5\da Z$ be a special birational transformation of type
$(4,b)$ and $n=3$. Then $b=1$ and one of the following holds:
\begin{enumerate}
\item[(i)] $d=12$, $g=19$, $\chi(\O_S)=16$, $\chi(\O_X)=-5$ and $z=4$;
\item[(ii)] $d=10$, $g=12$, $\chi(\O_S)=7$, $\chi(\O_X)=0$ and $z=6$;
\item[(iii)] $d=9$, $g=9$, $\chi(\O_S)=4$, $\chi(\O_X)=1$ and $z=8$.
\end{enumerate}
\end{proposition}

\begin{proof}
In this case, Propositions \ref{prop:ff} and \ref{prop:ff n=3} yield the fundamental formulae
$$z=874-180d+150g-14c_2(S)-c_3(X)$$
$$bz=234-42d+22g-c_2(S)$$
Now since $c_2(S)=12\chi(\O_S)-K_S^2$ (Noether's formula) and
$$c_3(X)=6K_S^2+24\chi(\O_X)-72\chi(\O_S)-2d(d-6-g)+12(g-1)$$
(this follows from the exact sequence $0\to TX\to T\p^r_{|X}\to N_{X/\p^r}\to 0$)
we deduce
$$z=-24\chi(\O_X)-96\chi(\O_S)+2d^2-2dg-192d+138g+8K_S^2+886$$
$$bz=-12\chi(\O_S)-42d+22g+K_S^2+234$$
Furthermore, since $m=2$ we get $e=0$ (see Remark \ref{rem:e}) and hence
$$b^2z=2g+62-8d$$
Moreover, we can eliminate variables thanks to Proposition \ref{prop:f in
P5} and we get
$$g=(b^2z+2(4d-31))/2$$
$$K_S^2=bz(6b-1)+d^2+d-128$$
$$\chi(\O_S)=(bz(17b-2)+d^2+47d-576)/12$$
$$K^3_X=z(2b^2(d-10)-2b+1)+2(3d^2-68d+384)$$
$$\chi(\O_X)=-(z(b^2(d+19)-8b+1)+6(d^2-9d-32))/24$$

Note that $S\subset\p^4$ has no $5$-secant lines as $X\subset\p^5$ is defined by quartic hypersurfaces. If $S\subset\p^4$ is a scroll then $X\subset\p^5$ is either the Segre embedding of $\p^1\times\p^2$, or a quadric hypersurface by Remark \ref{rem:scroll P4}. Therefore, $S\subset\p^4$ is a not scroll and we deduce $N_{5,4}=0$, that is,

\begin{center}
$2b^3z^2(b(3d-32)+2)-2bz(b(3d^3-103d^2+1185d-4568)+2(d^2-20d+106))+d^5-62d^4+1529d^3-18764d^2+114688d-279552=0$.
\end{center}

We deduce from Proposition \ref{prop:ff} that $d<16$. If $b=2$ then $i=5$ and hence $z=2$, so $d\notin\n$. On the other hand, if $b=1$ then $i=3$ and hence $z\in\{4,6,8,10,12,14,16,18\}$ by the classification of Mukai manifolds \cite{muk}. If $z=4$ then $d=12$ and we get (i). If $z=6$ then $d=10$ and we get (ii). If $z=8$ then $d=9$ and we get (iii). If $z=10$ (resp. $z=14$) then $d=8$ (resp. $d=7$) but $\chi(\O_X)\notin\z$. On the other hand, if $z\in\{12,16,18\}$ then $d\notin\n$.
\end{proof}

We will need the following result to obtain the main result of this section, namely, Theorem \ref{thm:n=3 r=5}:

\begin{lemma}\label{lem:6-secant}
Let $X\subset\p^5$ be a non-degenerate threefold. If $S\subset\p^4$ contains a line $L$ with $L^2\in\{0,-1\}$ then $X\subset\p^5$ is either the Segre embedding $\p^1\times\p^2$, or a c.i. $(2,2)$, or a quintic Castelnuovo threefold, or a scroll of lines over a surface (see \cite{ott} for the classification), or an inner projection of a c.i. $(2,2,2)$ in $\p^6$. In particular, $\Sec_5(X)=\emptyset$.
\end{lemma}

\begin{proof}
If $S\subset\p^4$ contains a line then $X\subset\p^5$ contains an irreducible family of lines $\mathcal F$ of dimension at least $2$. We can assume that $K_X+2H_X$ is generated by sections (see for instance \cite[Proposition 3.1]{b-o-s-s2}) and that $\dim(\mathcal F)=2$, as otherwise $X\subset\p^5$ is the Segre embedding of $\p^1\times\p^2$. If $\mathcal F$ sweeps out $X\subset\p^5$ then $N_{L/X}=\O_L\oplus\O_L$ for a general line $L\in\mathcal F$. Therefore $(K_X+2H_X)\cdot L=0$, so $\dim(\phi_{|K_X+2H_X|}(X))\leq 2$, and hence we deduce from the proof of \cite[Proposition 3.2]{b-o-s-s2} that $X\subset\p^5$ is a c.i. $(2,2)$, or a quintic Castelnuovo threefold, or a scroll of lines over a surface. On the other hand, if $\mathcal F$ does not sweep out $X\subset\p^5$ and $S\subset\p^4$ contains a line $L$ with $L^2\in\{0,-1\}$ then $X\subset\p^5$ contains a plane $P$ such that $N_{P/X}\cong\O_{\p^2}(-s)$ with $s\in\{0,1\}$. If $s=0$ then $X\subset\p^5$ would be the Segre embedding $\p^1\times\p^2$, which is a contradiction, and if $s=1$ then we get $d=7$ (see for instance \cite[Proposition 3.3]{b-o-s-s2})). In this case, $X\subset\p^5$ is necessarily an inner projection of a c.i. $(2,2,2)$ in $\p^6$ by \cite{ion}, since the other two types of threefolds in $\p^5$ of degree $7$ are easily ruled out.  As $X\subset\p^5$ is defined by forms of degree at most four in all of these cases, it does not have any $5$-secant line and hence $\Sec_5(X)=\emptyset$.
\end{proof}

\begin{proposition}[Maximal list for $n=3$ and $(a,b)=(5,b)$]\label{prop:(5,b)}
Let $\Phi:\p^5\da Z$ be a special birational transformation of type
$(5,b)$ and $n=3$. If $b>1$ then one of the following holds:
\begin{enumerate}
\item[(i)] $b=5$, $d=15$, $g=26$, $\chi(\O_S)=20$, $\chi(\O_X)=-4$ and $z=1$;
\item[(ii)] $b=3$, $d=14$, $g=22$, $\chi(\O_S)=14$, $\chi(\O_X)=0$ and $z=3$;
\item[(iii)] $b=2$, $d=16$, $g=28$, $\chi(\O_S)=16$, $\chi(\O_X)=10$ and $z=6$;
\item[(iv)] $b=2$, $d=12$, $g=16$, $\chi(\O_S)=9$, $\chi(\O_X)=0$ and $z=14$.
\end{enumerate}
\end{proposition}

\begin{proof}
In view of Propositions \ref{prop:ff} and \ref{prop:ff n=3}, we get the fundamental formulae
$$z=-24\chi(\O_X)-156\chi(\O_S)+2d^2-2d(g+246)+268g+13K_S^2+2857$$
$$bz=-12\chi(\O_S)-80d+30g+K_S^2+595$$
$$bz^2-e=123-11d+2g$$
Thanks to Proposition \ref{prop:f in P5} we can eliminate variables and we get
$$g=(b^2z+11d-e-123)/2$$
$$K_S^2=bz(10b-1)+d^2+25d-5(2e+125)$$
$$\chi(\O_S)=(bz(25b-2)+d^2+110d-25(e+75))/12$$
$$K_X^3=z(b^2(2d-2)-7b+1)+2(6d^2-d(e+148)+e+750)$$
$$\chi(\O_X)=-(z(b^2(d+61)-13b+1)+9d^2-de-61e-2625)/24$$
Moreover, we get that $S\subset\p^4$ has no $6$-secant lines since $X\subset\p^5$ is defined by quintic hypersurfaces, and we deduce from Lemma \ref{lem:6-secant} that it contains no line of self-intersection $\{0,-1\}$ as $\deg(\Sec_5(X))=5b-1$ by Proposition \ref{prop:sec}(i). Therefore, $N_{6,4}=0$ and we get the following equation on $d,e,b,z$

\begin{center}
$-3b^6z^3+b^2z^2(b^2(18d^2-615d+9e+5273)+4b(6d-95)+8)-bz(b(9d^4-658d^3+4d^2(9e+4520)-10d(123e+22114)+9e^2+10546e+1015245)+4(2d^3-103d^2+2d(3e+904)-5(19e+2145)))+d^6 -116d^5+d^4(9e+5588)-d^3(658e+143159)+2d^2(9e^2+9040e+1029038)-5d(123e^2+44228e+3149625)+3e^3+5273e^2+1015245e+50124375=0$
\end{center}

We get $d<25$ by Proposition \ref{prop:ff} and, as $Y\subset Z$ is defined by more than two forms of degree $b$, we deduce $e<zb^2$.
If $b=5$ then $i=6$ and hence $z=1$. Since $e<25$ we get
that either $d=e=15$ and we get case (i), or else $d=e=20$, but in
that case $\chi(\O_X)\notin\z$. If $b=4$ then $i=5$ and hence $z=2$.
Moreover $e<32$. In this case the only admissible
pairs are $(d,e)\in\{(13,16), (16,25), (17,24)\}$, but in any case
$\chi(\O_X)\notin\z$. Assume now $b=3$. Then $i=4$ and hence $Z$ is
a Del Pezzo manifold, so $z\in\{3,4,5\}$ by \cite{fuj}. If $z=3$ then $e<27$
and hence $(d,e)\in\{(14,14), (16,18), (17,17)\}$. In the first case
we get (ii), and in the other cases we get $\chi(\O_X)\notin\z$. If
$z=4$ then $e<36$ and hence $(d,e)\in\{(15,24), (16,33)\}$, but
$\chi(\O_X)\notin\z$. If $z=5$ then $e<45$, whence
$(d,e)\in\{(13,27), (22,36)\}$ and $\chi(\O_X)\notin\z$. Finally,
let $b=2$. Then $i=3$, so $Z$ is a Fano manifold of coindex $3$ and
hence $z\in\{4,6,8,10,12,14,16,18\}$ by Mukai's classification \cite{muk}. If $z=4$ then $e<16$ so we get $(d,e)\in\{(13,0), (16,9),
(17,8)\}$ and hence $\chi(\O_X)\notin\z$. If $z=6$ then $e<24$
so we get $(d,e)\in\{(15,12), (16,21)\}$. In the first case,
$\chi(\O_X)\notin\z$. In the second case, we get (iii). If $z=8$ then
$e<32$ and hence $(d,e)=(14,17)$, but $\chi(\O_X)\notin\z$. If
$z=10$ (resp. $z=12$) then $e<40$ (resp. $48$) and there are no
integer solutions $(d,e)$. If $z=14$ then $e<56$ an hence
$(d,e)\in\{(8,17),(12,33),(13,36),(16,39)\}$. If $(d,e)=(12,33)$
then we get (iv). In the other cases, $\chi(\O_X)\notin\z$. If $z=16$ then
$e<64$ and there are no integer solutions $(d,e)$. If $z=18$
then $e<72$ and hence $(d,e)=(18,69)$ but $\chi(\O_X)\notin\z$.
\end{proof}

\begin{remark}\label{rem:(5,b)}
The proof of Proposition \ref{prop:(5,b)} is based on the
classification of Fano manifolds of small coindex. If $b=1$ then $i=2$
and hence $Z$ has coindex $4$. As there is no
classification of such manifolds, we argue in a different way (cf. Remark \ref{rem:fano}).
\end{remark}

We will use the following results from \emph{liaison} theory in the sequel:

\begin{lemma}[{\cite[Propositions 3.1 and 2.5]{p-s}}]\label{lem:liaison}
Let $X, X'\subset\p^r$ be two subvarieties of codimension $2$ which are linked by a c.i. $(p,q)$. Let $d$ and $d'$ (resp. $g$ and $g'$) denote their degree (resp. sectional arithmetic genus). Then:
\begin{enumerate}
\item[(i)] $d+d'=pq$ and $g-g'=(p+q-4)(d-d')/2$
\item[(ii)] If ${\mathcal I}_X$ is given by a resolution  $0\to E\to F\to {\mathcal I}_X\to 0$ and $h^1(E(p))=h^1(E(q))=0$ then ${\mathcal I}_{X'}$ is given by a resolution $$0\to F^*(-p-q)\to E^*(-p-q)\oplus\O_{\p^r}(-p)\oplus\O_{\p^r}(-q)\to {\mathcal I}_{X'}\to 0$$
\end{enumerate}
\end{lemma}

\begin{lemma}[{\cite[Theorem 4.1]{p-s}}]\label{lem:smooth liaison}
Let $X\subset\p^r$ be a manifold of codimension $2$. If ${\mathcal I}_X(a)$ is globally generated then $X\subset\p^r$ can be linked to a manifold $X'\subset\p^r$ by a c.i. $(p,q)$ for every $p,q\geq a$.
\end{lemma}

\begin{proposition}[Maximal list for $n=3$ and $(a,b)=(5,1)$]\label{prop:(5,1)}
Let $\Phi:\p^5\da Z$ be a special birational transformation of type
$(5,1)$ and $n=3$. Then one of the following holds:
\begin{enumerate}
\item[(i)] $d=20$, $g=51$, $\chi(\O_S)=70$, $\chi(\O_X)=-55$ and $z=5$;
\item[(ii)] $d=17$, $g=33$, $\chi(\O_S)=23$, $\chi(\O_X)=9$ and $z=29$;
\item[(iii)] $d=17$, $g=35$, $\chi(\O_S)=34$, $\chi(\O_X)=-12$ and $z=13$;
\item[(iv)] $d=17$, $g=36$, $\chi(\O_S)=39$, $\chi(\O_X)=-21$ and $z=8$;
\item[(v)] $d=16$, $g=31$, $\chi(\O_S)=29$, $\chi(\O_X)=-11$ and $z=9$;
\item[(vi)] $d=15$, $g=27$, $\chi(\O_S)=23$, $\chi(\O_X)=-7$ and $z=12$;
\item[(vii)] $d=14$, $g=23$, $\chi(\O_S)=17$, $\chi(\O_X)=-3$ and $z=16$;
\item[(viii)] $d=13$, $g=19$, $\chi(\O_S)=11$, $\chi(\O_X)=1$ and $z=21$;
\item[(ix)] $d=12$, $g=15$, $\chi(\O_S)=6$, $\chi(\O_X)=2$ and $z=21$;
\item[(x)] $d=11$, $g=13$, $\chi(\O_S)=6$, $\chi(\O_X)=1$ and $z=42$;
\item[(xi)] $d=11$, $g=15$, $\chi(\O_S)=10$, $\chi(\O_X)=2$ and $z=68$.
\end{enumerate}
\end{proposition}

\begin{proof}
In this case $Z$ is a Fano manifold of coindex $4$, so we have no control on $z$ and hence we need a different approach. From Propositions \ref{prop:ff}, \ref{prop:ff n=3} and \ref{prop:f in P5} we can eliminate variables in terms of $d$ and $g$, and we get
$$e=z+11d-2g-123$$
$$K_S^2=-z+d^2-85d+5(4g+121)$$
$$\chi(\O_S)=-(2z-d^2+165d-50(g+24))/12$$
$$K_X^3=-6z-10d^2+d(4g-28)-2(2g-627)$$
$$\chi(\O_X)=(6z+d^2+d(274-g)-61g-2439)/12$$

Moreover, as $X\subset\p^5$ is defined by quintic hypersurfaces, we deduce that $S\subset\p^4$ has no $6$-secant lines and contains no line of self-intersection $\{0,-1\}$ by Lemma \ref{lem:6-secant}, since $\deg(\Sec_5(X))=4$ by Proposition \ref{prop:sec}(i)). So $N_{6,4}=0$ and we get the formula

\begin{center}
$8z^2-4z(2d^3-37d^2+d(25-12g)+10(19g+96))+d^6-17d^5-d^4(18g+579)+d^3(524g+17525)+2d^2(36g^2-1211g-78853)-4d(516g^2+10901g-136122)-8(3g^3-2083g^2-36438g+69768)=0$
\end{center}

First note that $d\leq 20$, as $X\subset\Sec_5(X)$ and $\deg(\Sec_5(X))=4$ (see Remark \ref{rem:sec} and Proposition \ref{prop:sec}(i)). If $d=20$ then $X\subset\p^5$ is a c.i. $(4,5)$ and $g=51$, whence either $z=5$, giving (i), or else $z=50$ and hence $\chi(\O_X)\notin\z$. If $d=19$ then $X\subset\p^5$ is linked $(4,5)$ to a $\p^3$ and hence $g=45$ by Lemma \ref{lem:liaison}(i), so either $z=41$, whence $\chi(\O_X)\notin\z$, or else $z=2$, which is not possible as the coindex of $Z$ is $4$. Now assume $d\leq 18$. We deduce that $\deg(Y)=e$ by Remark \ref{rem:e}, since $m=3$ by Proposition \ref{prop:num n=3}(v). As $b=1$, we get that $Y\subset Z$ is defined by more than two linear forms and hence we deduce $\deg(Y)=e<z$. Therefore $(11d-123)/2<g$. On the other hand, an upper bound on $g$ is given by \cite{g-p}. If $d=18$ then $38\leq g\leq40$. If $g\in\{38,40\}$ then $z\notin\n$. If $g=39$ then $z=36$. In this case, $X\subset\p^5$ is linked $(5,5)$ to a threefold $X'\subset\p^5$ with $d'=7$ and $g'=6$ by Lemmas \ref{lem:smooth liaison} and \ref{lem:liaison}(i). Therefore, $X'$ is linked to $\p^3$ by a c.i. $(2,4)$ (see for instance \cite{ion}) and we deduce from Lemma \ref{lem:liaison}(ii) that ${\mathcal I}_X$ is given by a resolution $0\to\O_{\p^5}(-8)\oplus\O_{\p^5}(-6)^{\oplus 2}\to\O_{\p^5}(-5)^{\oplus 4}\to {\mathcal I}_X\to 0$, so we get the contradiction $h^0({\mathcal I}_X(5))=4$. If $d=17$ then $33\leq g\leq36$. If $g=33$ then $z=29$, giving (ii). If $g=34$ then $z=21$, and hence $\chi(\O_X)\notin\z$. If $g=35$ then either $z=13$, giving (iii), or else $z=1$, which is not possible as the coindex of $Z$ is $4$. If $g=36$ then $z=8$, giving (iv). If $d=16$ then $27\leq g\leq33$. If $g=27$ then $z\notin\n$. If $g=28$ then $z=12$, and hence $\chi(\O_X)\notin\z$. If $g=29$ then either $z=10$, and hence $\chi(\O_X)\notin\z$, or else $z=1$, which is not possible. If $g=30$ then either $z=8$, and hence $\chi(\O_X)\notin\z$, or else $z=2$, which is not possible. If $g=31$ then $z=9$, giving (v). If $g=32$ then $z\notin\n$. If $g=33$ then $z=17$, and hence $\chi(\O_X)\notin\z$. If $d=15$ then $22\leq g\leq 28$, and $z\in\n$ only for $g\in\{26,27\}$. If $g=26$ then $z=5$ and hence $\chi(\O_X)\notin\z$. On the other hand, if $g=27$ then either $z=12$, giving (vi), or else $z=3$, which is not possible as the coindex of $Z$ is $4$. If $d=14$ then $16\leq g\leq 24$, and $z\in\n$ only for $22\leq g\leq 24$. If $g=22$ then $z\in\{6,9\}$, but in both cases $\chi(\O_X)\notin\z$. If $g=23$ then either $z=16$, giving (vii), or else $z=10$, giving the contradiction $e<0$. If $d=13$ then $11\leq g\leq 21$, and $z\in\n$ only for $18\leq g\leq21$. If $g=18$ then $z\in\{11,8\}$. In both cases $\chi(\O_X)\notin\z$. If $g=19$ then either $z=21$, giving (viii), or else $z=15$, giving the contradiction $e<0$. If $g=20$ then $z\in\{28,25\}$. In both cases $\chi(\O_X)\notin\z$. If $g=21$ then $z=35$, and hence $\chi(\O_X)\notin\z$. If $d=12$ then $5\leq g\leq 19$, and $z\in\n$ only for $15\leq g\leq19$. If $g=15$ then either $z=21$, giving (ix), or else $z=18$, and hence $\chi(\O_X)\notin\z$. If $g=16$ then $z\in\{38,24\}$. If $g=17$ then $z\in\{44,41\}$. If $g=18$ then $z=54$. If $g=19$ then $z\in\{70,61\}$. In all these cases $\chi(\O_X)\notin\z$. If $d=11$ then $g\leq 15$, and $z\in\n$ only for $g\in\{13,14,15\}$. If $g=13$ then either $z=42$, giving (x), or else $z=45$, and hence $\chi(\O_X)\notin\z$. If $g=14$ then $z\in\{61,55\}$ and in both cases $\chi(\O_X)\notin\z$. If $g=15$ then either $z=77$, and hence $\chi(\O_X)\notin\z$, or else $z=68$, giving (xi). Finally, if $d\leq 10$ we can argue in a similar way. But we conclude in view of \cite{b-s-s}.
\end{proof}

We analyze the cases in which either $X\subset\p^5$ does not exist, or $Z$ is singular in Propositions \ref{prop:(4,b)}, \ref{prop:(5,b)} and \ref{prop:(5,1)}.

\begin{lemma}\label{lem:(4,b)}
In Proposition \ref{prop:(4,b)}, $Z$ turns out to be singular in cases (i) and (ii). More precisely:
\begin{itemize}
\item In case (i), $X\subset\p^5$ is a c.i. $(3,4)$ and $Z\subset\p^6$ is a singular quartic hypersurface.
\item In case (ii), $X\subset\p^5$ is linked to a quadric hypersurface of $\p^4$ by a c.i. $(3,4)$ and $Z\subset\p^7$ is a singular c.i. $(2,3)$.
\end{itemize}
\end{lemma}

\begin{proof}
The classification of $X\subset\p^5$ follows from \cite{g-p} and \cite[Theorem 5.1.3]{b-s-s}, respectively. The description of $Z$ follows from Example \ref{ex:c.i} and Proposition \ref{prop:c.i}.
\end{proof}

\begin{lemma}\label{lem:(5,b)}
In Proposition \ref{prop:(5,b)}, case (iii) does not exist. In case (ii), ${\mathcal I}_X$ is given by a resolution $0\to T_{\p^5}(-7)\oplus\O_{\p^5}(-6)\to\O_{\p^5}(-5)^{\oplus 7}\to {\mathcal I}_X\to 0$ and $Z\subset\p^6$ is a singular cubic hypersurface.
\end{lemma}

\begin{proof}
In case (iii), we deduce from Lemmas \ref{lem:smooth liaison} and \ref{lem:liaison}(i) that $X\subset\p^5$ is linked $(5,5)$ to a threefold $X'\subset\p^5$ of degree $d'=9$ and sectional genus $g'=7$, contradicting \cite[Theorem 5.1]{b-s-s}. In case (ii), $X\subset\p^5$ is linked $(5,5)$ to a threefold $X'\subset\p^5$ of degree $d'=11$ and sectional genus $g'=13$. It follows from \cite[\S 4]{b-s-s2} and \cite[Remark 5.9]{d-p} that ${\mathcal I}_{X'}$ is given by a resolution $0\to\O_{\p^5}(-5)^{\oplus 5}\to \Omega_{\p^5}(-3)\oplus\O_{\p^5}(-4)\to {\mathcal I}_{X'}\to 0$, so we conclude by Lemma \ref{lem:liaison}(ii) and Example \ref{ex:hypersurfaces} and Remark \ref{rem:hypersurfaces}.
\end{proof}

\begin{lemma}\label{lem:(5,1)}
In Proposition \ref{prop:(5,1)}, cases (ii), (iii) and (xi) do not exist. Moreover, $Z$ is singular in cases (i), (iv), (v), (vi), (vii) and (ix). More precisely:
\begin{enumerate}
\item[(I)] In case (i), $X\subset\p^5$ is a c.i. $(4,5)$ and $Z\subset\p^6$ is a singular quintic hypersurface.
\item[(II)] In case (iv), $Z\subset\p^7$ is a singular c.i. $(2,4)$ and ${\mathcal I}_X$ is given by a resolution $0\to\O_{\p^5}(-6)\oplus\O_{\p^5}(-8)\to\O_{\p^5}(-4)\oplus\O_{\p^5}(-5)^{\oplus 2}\to {\mathcal I}_X\to 0$.
\item[(III)] In case (v), $Z\subset\p^7$ is a singular c.i. $(3,3)$ and ${\mathcal I}_X$ is given by a resolution $0\to\O_{\p^5}(-7)^{\oplus 2}\to\O_{\p^5}(-4)\oplus\O_{\p^5}(-5)^{\oplus 2}\to {\mathcal I}_X\to 0$.
\item[(IV)] In case (vi), $Z\subset\p^8$ is a singular c.i. $(2,2,3)$ and ${\mathcal I}_X$ is given by a resolution $0\to\O_{\p^5}(-6)^{\oplus 2}\oplus\O_{\p^5}(-7)\to\O_{\p^5}(-4)\oplus\O_{\p^5}(-5)^{\oplus 3}\to {\mathcal I}_X\to 0$.
\item[(V)] In case (vii), $Z\subset\p^9$ is a singular c.i. $(2,2,2,2)$ and ${\mathcal I}_X$ is given by a resolution $0\to\O_{\p^5}(-6)^{\oplus 4}\to\O_{\p^5}(-4)\oplus\O_{\p^5}(-5)^{\oplus 4}\to {\mathcal I}_X\to 0$.
\item[(VI)] In case (ix), $Z\subset\p^9$ is a singular variety of degree $21$ and ${\mathcal I}_X$ is given by a resolution $0\to T_{\p^5}(-6)\oplus\O_{\p^5}(-5)^{\oplus 4}\to\Omega^2_{\p^5}(-2)\to {\mathcal I}_X\to 0$.
\end{enumerate}
\end{lemma}

\begin{proof}
If the numerical invariants of $X\subset\p^5$ are those of Proposition \ref{prop:(5,1)}(ii) (resp. (iii)), then we get from Lemmas \ref{lem:smooth liaison} and \ref{lem:liaison}(i) that $X\subset\p^5$ is linked $(5,5)$ to a threefold $X'\subset\p^5$ of degree $d'=8$ and sectional genus $g'=6$ (resp. $8$). But such a threefold $X'\subset\p^5$ does not exist according to \cite{oko}. On the other hand, the numerical invariants of Proposition \ref{prop:(5,1)}(xi) are ruled out by \cite{b-s-s2} and \cite[Remark 5.9]{d-p}. Assume now that the invariants of $X\subset\p^5$ are given by Proposition \ref{prop:(5,1)}(i). Then $X\subset\p^5$ is a c.i. $(4,5)$ by \cite{g-p}, and $Z\subset\p^6$ is a singular quintic hypersurface by Example \ref{ex:c.i} and Proposition \ref{prop:c.i}. This gives (I). If the invariants of $X\subset\p^5$ are those of Proposition \ref{prop:(5,1)}(iv), then it is linked $(5,5)$ to a threefold $X'\subset\p^5$ of degree $d'=8$ and $g'=9$. Then $X'\subset\p^5$ is a c.i. $(2,4)$ by \cite{oko}. The resolution of ${\mathcal I}_X$ is given by Lemma \ref{lem:liaison}(ii), and $Z\subset\p^7$ is a singular c.i. $(2,4)$ by Example \ref{ex:c.i} and Proposition \ref{prop:c.i}. So we get (II). If $X\subset\p^5$ has invariants as in Proposition \ref{prop:(5,1)}(v) then it is linked $(5,5)$ to a threefold of degree $d'=9$ and sectional genus $g'=10$. Then $X'\subset\p^5$ is a c.i. $(3,3)$ by \cite{b-s-s}. The resolution of ${\mathcal I}_X$ is given by Lemma \ref{lem:liaison}(ii), and $Z\subset\p^7$ is a singular c.i. $(3,3)$ by Example \ref{ex:c.i} and Proposition \ref{prop:c.i}. This gives (III). If the invariants of $X\subset\p^5$ are given by Proposition \ref{prop:(5,1)}(vi) then it is linked $(5,5)$ to a threefold $X'\subset\p^5$ of degree $d'=10$ and sectional genus $g'=12$. Therefore ${\mathcal I}_{X'}$ is given by a resolution $0\to\O_{\p^5}(-5)\oplus\O_{\p^5}(-6)\to\O_{\p^5}(-3)\oplus\O_{\p^5}(-4)^{\oplus 2}\to {\mathcal I}_{X'}\to 0$ by \cite[Theorem 5.1]{b-s-s}. So we get the resolution of ${\mathcal I}_X$ by Lemma \ref{lem:liaison}(ii), and $Z\subset\p^8$ is a singular c.i. $(2,2,3)$ by Example \ref{ex:c.i} and Proposition \ref{prop:c.i}. This gives (IV). Let $X\subset\p^5$ be with numerical invariants as in Proposition \ref{prop:(5,1)}(vii), so it is linked $(5,5)$ to a threefold $X'\subset\p^5$ of degree $d'=11$ and sectional genus $g'=14$. It follows from \cite{b-s-s2} and \cite[Example 5.8]{d-p} that $X'\subset\p^5$ is linked $(4,4)$ to a quintic threefold, so ${\mathcal I}_{X'}$ is given by a resolution $0\to\O_{\p^5}(-5)^{\oplus 2}\oplus\O_{\p^5}(-6)\to\O_{\p^5}(-4)^{\oplus 4}\to {\mathcal I}_{X'}\to 0$. Then the resolution of ${\mathcal I}_X$ follows from Lemma \ref{lem:liaison}(ii) so $Z\subset\p^9$ is a singular c.i. $(2,2,2,2)$ by Example \ref{ex:c.i} and Proposition \ref{prop:c.i}, giving (V). Finally, if the invariants of $X\subset\p^5$ are those of Proposition \ref{prop:(5,1)}(ix) then $(K_X+H_X)^3=0$. So it is a conic bundle and the resolution of ${\mathcal I}_X$ is given by \cite[Propositions 3.4 and 5.9]{b-o-s-s2}, respectively. To get a contradiction, assume that $Z\subset\p^9$ is smooth. Then $Y\subset Z$ has codimension $2$ by Proposition \ref{prop:num n=3}(v) and, on the other hand, the formula $e=z+11d-2g-123$ given in Poposition \ref{prop:(5,1)} yields $e=0$, contradicting Remark \ref{rem:e} and giving (VI).
\end{proof}

\begin{remark}\label{rem:sing}
In all the cases where $Z$ turns out to be singular throughout the paper, the map $\tau:W\to Z$ is a divisorial contraction. Hence, in all these cases, $Z$ is a normal variety with only $\q$-factorial and terminal singularities. As Lemma \ref{lem:(5,1)}(VI) shows, Propositions \ref{prop:K} and \ref{prop:sec} are no longer true in this setting. In this case, $H=H_Z-\frac{1}{2}E_Z$, $K_W=-2H_Z+\frac{1}{2}E_Z$ and $\Sec_5(X)\subset\p^5$ is a hypersurface of degree $8$. This shows that the classification of special birational transformations of $\p^r$ is more complicated if one allows (even mild) singularities on $Z$.
\end{remark}

The proof of Theorem \ref{thm:n=3 r=5} uses the classification of threefolds $X\subset\p^5$ of low degree. This classification is complete up to degree $11$ (see \cite{b-s-s2} and \cite[Remark 5.9]{d-p}). For degree $12$, the possible numerical invariants of $X\subset\p^5$ are given in \cite{ede} but uniqueness is not known (see \cite[p. 394]{ede} and cf. Remark \ref{rem:linked}). So we provide a direct argument in the following two cases:

\begin{lemma}\label{lemma:12}
Let $X\subset\p^5$ be as in Proposition \ref{prop:(5,b)}(iv). Then ${\mathcal I}_X$ is given by a resolution
$0\to\O_{\p^5}(-5)^{\oplus 3}\oplus\O_{\p^5}(-6)\to\Omega_{\p^5}(-3)\to {\mathcal I}_X\to 0$.
\end{lemma}

\begin{proof}
Let $S\subset\p^4$ be a general hyperplane section of $X\subset\p^5$, and let $C\subset\p^3$ be a general hyperplane section of $S\subset\p^4$. Since $h^0(\O_{\p^3}(4))=35$ and $h^0(\O_C(4))=33$ we deduce $h^0({\mathcal I}_C(4))\geq 2$, i.e. $C$ is contained in two quartic surfaces of $\p^3$. Let $C'$ denote the residual scheme of degree $d'=4$. We deduce from Lemma \ref{lem:liaison}(i) that $p_a(C')=0$. Therefore $h^0({\mathcal I}_{C'}(2))=1$, and hence $C'$ is a divisor of type $(3,1)$ in a smooth quadric of $\p^3$. So ${\mathcal I}_{C'}$ is given by a resolution $0\to T_{\p^3}(-5)\to\O_{\p^3}(-3)^{\oplus 3}\oplus\O_{\p^3}(-2)\to {\mathcal I}_{C'}\to 0$, and we deduce from Lemma \ref{lem:liaison}(ii) that ${\mathcal I}_C$ is given by a resolution $0\to\O_{\p^3}(-5)^{\oplus3}\oplus\O_{\p^3}(-6)\to\Omega_{\p^3}(-3)\oplus\O_{\p^3}(-4)^{\oplus 2}\to {\mathcal I}_C\to 0$.

Let us compute the Beilinson cohomology table of ${\mathcal I}_S(4)$, i.e. the $5\times5$ matrix with entries $h^{4-i}({\mathcal I}_S(j))$. As $X\subset\p^5$ is regular by Barth's theorem, $S\subset\p^4$ is regular by Kodaira vanishing. Then $h^1(\O_S)=0$, and hence $h^2(\O_S)=8$ by Riemann-Roch. Note that Serre's duality yields $h^2(\O_S(2))=h^0(K_S-2H_S)=0$, since $(K_S-2H_S)H_S=2g-2-3d=-6$. Therefore $h^3({\mathcal I}_S(2))=h^2(\O_S(2))=0$, and hence we deduce from the exact sequence $0\to {\mathcal I}_S(k-1)\to {\mathcal I}_S(k)\to {\mathcal I}_C(k)\to 0$ that the Beilinson cohomology table of ${\mathcal I}_S(4)$ is given by
\[
\begin{array}{ccccc}
0 & 0 & 0 & 0 & 0 \\
8 & 1 & 0 & 0 & 0 \\
0 & 0 & 0 & 0 & 0 \\
0 & 0 & 0 & 1 & s \\
0 & 0 & 0 & 0 & s+1 \\
\end{array}
\]
Now let us compute the Beilinson cohomology table of ${\mathcal I}_X(4)$, i.e. the $6\times6$ matrix with entries $h^{5-i}({\mathcal I}_X(j-1))$. The Kodaira vanishing theorem yields $h^i(\O_X(-1))=0$ for $i\in\{0,1,2\}$, and hence $h^3(\O_X(-1))=9$ by Riemann-Roch. Furthermore, $h^3(\O_X(1))=h^0(K_X-H_X)=0$ since $(K_X-H_X)H^2_X=2g-2-3d=-6$. Therefore $h^4({\mathcal I}_X(1))=h^3(\O_X(1))=0$, and hence we deduce from the exact sequence $0\to {\mathcal I}_X(k-1)\to {\mathcal I}_X(k)\to {\mathcal I}_S(k)\to 0$ that the Beilinson cohomology table of ${\mathcal I}_X(4)$ is given by
\[
\begin{array}{cccccc}
0 & 0 & 0 & 0 & 0 & 0 \\
9 & 1 & 0 & 0 & 0 & 0 \\
0 & 0 & 0 & 0 & 0 & 0 \\
0 & 0 & 0 & 0 & 0 & 0 \\
0 & 0 & 0 & 0 & 1 & 0 \\
0 & 0 & 0 & 0 & 0 & 0 \\
\end{array}
\]
as $h^0({\mathcal I}_X(4))=0$ by Proposition \ref{prop:sec}(i). Consequently, ${\mathcal I}_X(4)$ is given by a resolution $0\to\O_{\p^5}(-1)^{\oplus 9}\to T_{\p^5}(-2)\oplus\Omega_{\p^5}(1)\to {\mathcal I}_X(4)\to 0$ by Beilinson's theorem \cite{bei}. Finally, using Euler's sequence $0\to\O_{\p^5}(-2)\to\O_{\p^5}(-1)^{\oplus 6}\to T_{\p^5}(-2)\to 0$, we get the diagram
\[
\xymatrix
{
         & 0 \ar[d]                                        & 0 \ar[d]                                       &                                     &   \\
0 \ar[r] & \O_{\p^5}(-1)^{\oplus 3}\oplus\O_{\p^5}(-2) \ar[r]\ar[d] & \Omega_{\p^5}(1)\ar[r]\ar[d]          & {\mathcal I}_X(4)\ar[r]\ar@{=}[d]              & 0 \\
0 \ar[r] & \O_{\p^5}(-1)^{\oplus 9}\ar[r]\ar[d]            & T_{\p^5}(-2)\oplus\Omega_{\p^5}(1)\ar[r]\ar[d] & {\mathcal I}_X(4)\ar[r]                        & 0 \\
         & T_{\p^5}(-2)\ar[d]\ar@{=}[r]                    & T_{\p^5}(-2)\ar[d]                             &                                     &   \\
         & 0                                               & 0                                              &                                     &   \\
}
\]
and the first line gives the resolution of the statement.
\end{proof}

\begin{lemma}\label{lemma:13}
Let $X\subset\p^5$ be as in Proposition \ref{prop:(5,1)}(viii). Then ${\mathcal I}_X$ is given by a resolution
$0\to\O_{\p^5}(-5)^{\oplus 10}\to\Omega^3_{\p^5}(-1)\oplus\O_{\p^5}(-4)\to {\mathcal I}_X\to 0$.
\end{lemma}

\begin{proof}
We argue as in Lemma \ref{lemma:12}. Since $X$ is contained in the quartic hypersurface $\Sec_5(X)\subset\p^5$ (see Remark \ref{rem:sec} and Proposition \ref{prop:sec}(i)), the curve $C\subset\p^3$ is linked $(4,5)$ to a curve $C'\subset\p^3$ of degree $d'=7$ and arithmetic genus $p_a(C')=4$ by Lemma \ref{lem:liaison}(i). Then $h^0({\mathcal I}_{C'}(3))\geq 2$, as $h^0(\O_{\p^3}(3))=20$ and $h^0(\O_{C'}(3))=18$. Let $C''\subset\p^3$ denote its residual scheme in a c.i. $(3,3)$. Then $d''=2$ and $p_a(C'')=-1$, so $C''\subset\p^3$ consists of two skew lines. Therefore, ${\mathcal I}_{C'}$ is given by a resolution $0\to T_{\p^3}(-6)\to\O_{\p^3}(-4)^{\oplus 2}\oplus\O_{\p^3}(-3)^{\oplus 2}\to {\mathcal I}_{C'}\to 0$ and ${\mathcal I}_C$ is given by a resolution $0\to\O_{\p^3}(-5)\oplus\O_{\p^3}(-6)^{\oplus 2}\to\Omega_{\p^3}(-3)\oplus\O_{\p^3}(-4)\to {\mathcal I}_C\to 0$ by Lemma \ref{lem:liaison}(ii).

Let us compute the Beilinson cohomology table of ${\mathcal I}_S(4)$. Since $S\subset\p^4$ is regular we get $h^1(\O_S)=0$, and hence $h^2(\O_S)=10$ by Riemann-Roch. Note that $h^2(\O_S(2))=h^0(K_S-2H_S)=0$, since $(K_S-2H_S)H_S=-3$. Hence $h^3({\mathcal I}_S(2))=h^2(\O_S(2))=0$. Furthermore, $h^1({\mathcal I}_C(j))=h^2({\mathcal I}_C(j))=0$ for every integer $j\geq 4$ so the Serre vanishing theorem yields $h^2({\mathcal I}_S(3))=0$, and hence we deduce from the exact sequence $0\to {\mathcal I}_S(k-1)\to {\mathcal I}_S(k)\to {\mathcal I}_C(k)\to 0$ that the Beilinson cohomology table of ${\mathcal I}_S(4)$ is given by
\[
\begin{array}{ccccc}
0 & 0 & 0 & 0 & 0 \\
10& s+1 & 0 & 0 & 0 \\
0 & s & 1 & 0 & 0 \\
0 & 0 & 0 & 0 & 0 \\
0 & 0 & 0 & 0 & 1 \\
\end{array}
\]

Now let us compute the Beilinson cohomology table of ${\mathcal I}_X(4)$. The Kodaira vanishing theorem yields $h^i(\O_X(-1))=0$ for $i\in\{0,1,2\}$, and hence $h^3(\O_X(-1))=10$ by Riemann-Roch. On the other hand, $h^3(\O_X(1))=h^0(K_X-H_X)=0$ as $(K_X-H_X)H_X^2=-3$. So $h^4({\mathcal I}_X(1))=0$. Furthermore $h^1({\mathcal I}_S(j))=h^2({\mathcal I}_S(j))=0$ for every integer $j\geq 3$, so the Serre vanishing theorem yields $h^2({\mathcal I}_X(2))=0$ and hence we deduce from the exact sequence $0\to {\mathcal I}_X(k-1)\to {\mathcal I}_X(k)\to {\mathcal I}_S(k)\to 0$ that the Beilinson cohomology table of ${\mathcal I}_X(4)$ is given by
\[
\begin{array}{cccccc}
0 & 0 & 0 & 0 & 0 & 0 \\
10& x & 0 & 0 & 0 & 0 \\
0 & x & 1 & 0 & 0 & 0 \\
0 & 0 & 0 & 0 & 0 & 0 \\
0 & 0 & 0 & 0 & 0 & 0 \\
0 & 0 & 0 & 0 & 0 & 1 \\
\end{array}
\]
Let us show that $x=0$. Assume to the contrary that $x=h^4({\mathcal I}_X)=h^3(\O_X)=h^0(K_X)\neq 0$. Then $|K_X+H_X|$ gives a birational map $\varphi_{|K_X+H_X|}:X\da\p^9$ onto its image. Since $\deg(\varphi(X))=(K_X+H_X)^3=14$, we get a contradiction by \cite[Corollary 2.23]{g-h}. For a similar proof, see the proof of \cite[Proposition 4.1]{ede} (and hence also \cite[Lemma 4.2]{b-s-s2}) having in mind that $X\subset\p^5$ coincides with its first reduction by \cite[Theorem 4.4.1]{b-s-s2}. Consequently, ${\mathcal I}_X(4)$ is given by a resolution $0\to\O_{\p^5}(-1)^{\oplus 10}\to\Omega^3_{\p^5}(3)\oplus\O_{\p^5}\to {\mathcal I}_X(4)\to 0$.
\end{proof}

\begin{remark}\label{rem:linked}
The threefolds of Lemmas \ref{lemma:12} and \ref{lemma:13} are linked by a c.i. $(5,5)$ by Lemma \ref{lem:liaison}(ii) and the end of Example \ref{ex:degree21}. But we cannot deduce Lemma \ref{lemma:13} from Lemma \ref{lemma:12} (and viceversa) since the uniqueness of threefolds with these numerical invariants was left open in \cite{ede}. Moreover, note that in Lemma \ref{lemma:12} we use the fact $h^0({\mathcal I}_X(4))=0$ coming from Proposition \ref{prop:sec}(i).
\end{remark}

We are now in position to prove Theorem \ref{thm:n=3 r=5}.

\begin{theorem}\label{thm:n=3 r=5}
Let $\Phi:\p^5\da Z$ be a special birational transformation of type
$(a,b)$ and $n=3$. Then one of the following holds:
\begin{enumerate}
\item[(I)] $(a,b)=(5,5)$, $Z=\p^5$ and ${\mathcal I}_X$ is given by a resolution $$0\to\O_{\p^5}(-6)^{\oplus 5}\to\O_{\p^5}(-5)^{\oplus 6}\to {\mathcal I}_X\to 0;$$
\item[(II)] $(a,b)=(5,2)$, $Z\subset\p^{11}$ is a linear section of $\g(1,5)\subset\p^{14}$ and ${\mathcal I}_X$ is given by a resolution $$0\to\O_{\p^5}(-5)^{\oplus 3}\oplus\O_{\p^5}(-6)\to\Omega_{\p^5}(-3)\to {\mathcal I}_X\to 0;$$
\item[(III)] $(a,b)=(5,1)$, $Z\subset\p^{10}$ is a new Fano manifold of degree $21$ and coindex $4$, and ${\mathcal I}_X$ is given by a resolution $$0\to\O_{\p^5}(-5)^{\oplus 10}\to\Omega^3_{\p^5}(-1)\oplus\O_{\p^5}(-4)\to {\mathcal I}_X\to 0;$$
\item[(IV)] $(a,b)=(5,1)$, $Z\subset\p^{15}$ is a linear section of $\g(1,6)\subset\p^{20}$ and ${\mathcal I}_X$ is given by a resolution $$0\to\O_{\p^5}(-5)^{\oplus 5}\to\Omega_{\p^5}(-3)\oplus\O_{\p^5}(-4)\to {\mathcal I}_X\to 0;$$
\item[(V)] $(a,b)=(4,1)$, $Z\subset\p^8$ is a c.i. $(2,2,2)$ and ${\mathcal I}_X$ is given by a resolution $$0\to\O_{\p^5}(-5)^{\oplus 3}\to\O_{\p^5}(-3)\oplus\O_{\p^5}(-4)^{\oplus 3}\to {\mathcal I}_X\to 0;$$
\item[(VI)] $(a,b)=(3,1)$, $Z\subset\p^7$ is a c.i. $(2,2)$ and ${\mathcal I}_X$ is given by a resolution $$0\to\O_{\p^5}(-4)^{\oplus 2}\to\O_{\p^5}(-2)\oplus\O_{\p^5}(-3)^{\oplus 2}\to {\mathcal I}_X\to 0;$$
\item[(VII)] $(a,b)=(2,1)$, $X\subset\p^4$ is a quadric hypersurface and $Z\subset\p^6$ is a quadric hypersurface.
\end{enumerate}
\end{theorem}

\begin{proof}
We proceed in view of Proposition \ref{prop:num n=3}, cases (v)-(viii). According to Lemmas \ref{lem:(5,b)}, \ref{lem:(5,1)} and \ref{lem:(4,b)}, we first analyze the remaining cases in Proposition \ref{prop:(5,b)}, \ref{prop:(5,1)} and \ref{prop:(4,b)}, respectively. Let $X\subset\p^5$ be with numerical invariants as in Proposition \ref{prop:(5,b)}(i). Then it is linked $(5,5)$ to a threefold $X'\subset\p^5$ of degree $d'=10$ and sectional genus $g'=11$ by Lemmas \ref{lem:smooth liaison} and \ref{lem:liaison}(i). So ${\mathcal I}_{X'}$ is given by a resolution $0\to\O_{\p^5}(-5)^{\oplus 4}\to\O_{\p^5}(-4)^{\oplus 5}\to {\mathcal I}_{X'}\to 0$ by \cite[Theorem 5.1.4]{b-s-s} and the resolution of ${\mathcal I}_X$ follows from Lemma \ref{lem:liaison}(ii), giving (I). In this case, we obtain a well-known Cremona transformation (see Remark \ref{rem:cremona}(ii)). Assume now that the invariants of $X\subset\p^5$ are given in Proposition \ref{prop:(5,b)}(iv). The resolution of ${\mathcal I}_X$ is given by Lemma \ref{lemma:12}, and $Z\subset\p^{11}$ is a linear section of $\g(1,5)\subset\p^{14}$ (see Example \ref{ex:g(1,r)}). This gives (II). If the invariants of $X\subset\p^5$ are those of Proposition \ref{prop:(5,1)}(viii), the resolution of ${\mathcal I}_X$ follows from Lemma \ref{lemma:13} and $Z\subset\p^{10}$ is a new Fano manifold of coindex $4$ (see Example \ref{ex:degree21}). This gives (III). If $X\subset\p^5$ has invariants as in Proposition \ref{prop:(5,1)}(x), then the resolution of ${\mathcal I}_X$ is given by \cite{b-s-s2} and \cite[Remark 5.9]{d-p}. In that case, $Z\subset\p^{15}$ is a linear section of $\g(1,6)\subset\p^{20}$ (see Example \ref{ex:g(1,r+1)}) and we get (IV). If the invariants of $X\subset\p^5$ are given in Proposition \ref{prop:(4,b)}(iii) then the resolution of ${\mathcal I}_X$ comes from \cite[Theorem 5.1.3]{b-s-s}. Hence $Z\subset\p^8$ is a c.i. $(2,2,2)$ by Example \ref{ex:c.i} and Proposition \ref{prop:c.i}, giving (V). On the other hand, cases (VI)-(VII) follow from Theorem \ref{thm:n=3 easy}.
\end{proof}

\begin{corollary}\label{cor:r=n+2}
Let $\Phi:\p^r\da Z$ be a special birational transformation of type $(a,b)$. If $r=n+2$ then either $r\in\{3,4,5\}$ and $X\subset\p^r$ and $Z$ are given in Theorem \ref{thm:n=1}, cases (IV)-(VII), Theorem \ref{thm:n=2}, cases (VI)-(XII), and Theorem \ref{thm:n=3 r=5}, cases (I)-(VI), or else $X\subset\p^{r-1}$ is a quadric hypersurface and $Z\subset\p^{r+1}$ is a quadric hypersurface.
\end{corollary}

\begin{proof}
If $r\leq 5$ we conclude by Theorems \ref{thm:n=1}, \ref{thm:n=2} and \ref{thm:n=3 r=5}. Assume $r\geq 6$. Note that $d<a^2=(m+2)^2$ by Propositions \ref{prop:ff} and \ref{prop:num}(iii), whence $d<r^2$. Therefore, we deduce from \cite[Theorem 5.1(ii)]{h-s} that $X\subset\p^r$ is a c.i. $(a',a)$, with $a'<a$. In that case $a'=\deg(\Sec_a(X))$, so Proposition \ref{prop:sec}(i) yields $b=1$ and $a'=a-1$. Therefore, we deduce from Example \ref{ex:c.i} and Proposition \ref{prop:c.i} that $a=2$.
\end{proof}

\begin{remark}\label{rem:h-r}
In view of Corollary \ref{cor:r=n+2}, we point out that if $r=n+2$ then the Hartshorne-Rao modules of $X\subset\p^r$ are very special. More precisely, $X\subset\p^r$ is either a complete intersection, or arithmetically Cohen-Macaulay, or arithmetically Buchsbaum (see \cite{chang}).
\end{remark}

In particular, we recover the following result:

\begin{corollary}[Ein and Shepherd-Barron]
Let $\Phi:\p^r\da\p^r$ be a special Cremona transformation of type
$(a,b)$. If $r=n+2$ then either $r\in\{3,4,5\}$, $(a,b)=(r,r)$, and $X\subset\p^r$ is defined by the $r\times r$-minors of an $r\times
(r+1)$-matrix of linear forms, or else $r=4$, $(a,b)=(3,2)$ and
$X\subset\p^4$ is a quintic elliptic scroll.
\end{corollary}

Moreover, Corollary \ref{cor:r=n+2} yields the following classification in the case $m=2$ (cf. Theorems \ref{thm:m=0} and \ref{thm:m=1}):

\begin{theorem}\label{thm:m=2}
Let $\Phi:\p^r\da Z$ be a special birational transformation of type
$(a,b)$. If $m=2$ then either $a=4$, $r=n+2$ and $X\subset\p^r$ and $Z$ are given in Theorem \ref{thm:n=2}, cases (VI)-(X), and Theorem \ref{thm:n=3 r=5}, case (VI), or else $a=2$, $r=n+3$ and $X\subset\p^r$ and $Z$ are given in Theorem \ref{thm:n=1}, cases (I)-(III), Theorem \ref{thm:n=2}, cases (IV)-(V), and Theorem \ref{thm:n=3 easy}, case (II).
\end{theorem}

\begin{proof}
If $m=2$ we deduce from Proposition \ref{prop:num}(iii) that either $a=4$ and $r=n+2$, or else $a=2$ and $r=n+3$. In the first case, we conclude by Corollary \ref{cor:r=n+2}. In the second case, Proposition \ref{prop:num}(iv) yields $b=(i-2)/(r-3)\leq (r-1)/(r-3)$. So $b\in\{1,2,3\}$
for $r=4$, $b\in\{1,2\}$ for $r=5$, and $b=1$ for $r\geq 6$. If $r=4$ we apply Theorem \ref{thm:n=1}, and if $r=5$ we apply Theorem \ref{thm:n=2}. Finally, if $r\geq 6$ then $b=1$ so $\deg(\Sec_2(X))=1$ by Proposition \ref{prop:sec}(i). In particular, $X\subset\p^r$ is contained in a hyperplane and $d\leq 4$ so we deduce that the only possibility is that given in Theorem \ref{thm:n=3 easy}(II).
\end{proof}

\section{Constructing examples}\label{section:ex}

In this section, we present some series of examples of birational transformations of $\p^r$. They contain, in particular, all the special birational transformations that appear throughout the paper. The following result, essentially noticed in \cite{kleiman}, will be useful for our purpose:

\begin{theorem}\label{thm:kleiman}
Let $E, F$ be vector bundles on $\p^r$, $r\leq 5$. If $\rk(F)=\rk(E)+1$ and $E^*\otimes F$ is globally generated, then the generic morphism $u: E\to F$ degenerates on a manifold $X\subset\p^r$ of codimension $2$. Furthermore, we get the exact sequence $$0\to E\stackrel{u}\to F\to {\mathcal I}_X(c_1(F)-c_1(E))\to 0$$
\end{theorem}

\begin{remark}\label{rem:porteous}
In the above context, if $r\geq 6$ the cycle corresponding to the singular locus of $X\subset\p^r$ is given by the well-known Giambelli-Thom-Porteous formula (see for instance \cite[Ch. 14]{ful}).
\end{remark}

\begin{example}\label{ex:c.i}
Let $X\subset\p^r$ be given by the resolution
$$0\to\oplus_{j=1}^c\O_{\p^r}(a-\beta_j)\stackrel{u}\to\O_{\p^r}^{\oplus c}\oplus\O_{\p^r}(1)\to {\mathcal I}_X(a)\to 0$$
corresponding to a morphism $u$, where $\sum_{j=1}^c\beta_j=(c+1)a-1$, $a\leq r$ and $\beta_j>a$. Then we claim that a basis of $H^0({\mathcal I}_X(a))$ defines a birational transformation $\Phi:\p^r\da Z\subset\p^{r+c}$ of type $(a,1)$ onto a c.i. $(\beta_1+1-a,\dots,\beta_c+1-a)$. To prove the claim, let $\p^r:=\p(V)$, with coordinates $\underline{x}:=(x_0:\dots:x_r)$, and let $\varepsilon_j:=\beta_{j}-a$. Note that $I_X$ is generated by a unique polynomial $f$ of degree $a-1$, corresponding to the equation of $\Sec_{a}(X)$, and $c$ polynomials $f_1,\dots,f_c$ of degree $a$, as we see directly from the matrix representing $u$ and the locus where the rank drops. By choosing coordinates $\underline{z}:=(z_0:\dots:z_r:z_{r+1}:\dots:z_{r+c})$ in $\p^{r+c}$, we have that $\Phi$ is given by $z_i=x_if$ for $i=0,\dots,r$ and $z_{r+j}=f_j$ for $j=1,\dots,c$. Note that the equations of $Y\subset\p^{r+c}$ are $\{z_0=\dots=z_r=0\}$ and that $\Phi_{|\p^r\backslash\Sec_a(X)}:\p^r\backslash\Sec_a(X)\to Z\backslash Y$ is an isomorphism. Let us consider the following diagram
\[
\xymatrix
{
        &                                                         & 0                                                  &                         &   \\
0\ar[r] & \oplus_{j=1}^c\O_{\p^r}(-\varepsilon_j)\ar[r]\ar@{=}[d] & \O_{\p^r}^{\oplus c}\oplus\O_{\p^r}(1)\ar[r]\ar[u] & {\mathcal I}_X(a)\ar[r] & 0 \\
        & \oplus_{j=1}^c\O_{\p^r}(-\varepsilon_j)\ar[r]           & \O_{\p^r}^{\oplus c}\oplus V\otimes\O_{p^r}\ar[u]  &                         &   \\
        &                                                         & \wedge^2(V)\otimes\O_{\p^r}(-1)\ar[u]              &                         &   \\
}
\]

By the mapping cone technique, we have that a set of $r+c+1$
generators of degree $a$ of $I_X$ is given by the unique syzygy of $M^{t}$, where $M$ is a
matrix as follows

\[
M=\left(
\begin{array}{cc}
\begin{array}{c}
(r+1)\times c\text{ matrix having} \\
\text{columns of degrees $\varepsilon_1,\dots,\varepsilon_c$}
\end{array}
&
\begin{array}{c}
(r+1)\times\binom{r+1}{2}\text{ Koszul matrix} \\
\text{$K$ in the $r+1$ variables $\underline{x}$}
\end{array}
\\
\begin{array}{c}
c\times c\text{ matrix having} \\
\text{columns of degrees $\varepsilon_1,\dots,\varepsilon_c$}
\end{array}
& 0
\end{array}
\right)
\]

Moreover, the first syzygies of this set of generators are the columns of $M$. Hence we have that the graph of $\Phi$ is given by $\underline{z}M=0$,
and a set of generators for $I_Z$ is obtained by eliminating $\underline{x}$ from $\underline{z}M=0$. In this case, the product $(z_0,\dots,z_r)K=0$ gives $x_i=z_i$ for $i=0,\dots, r$. Therefore, the elimination is straightforward and $I_Z$ is generated by exactly $c$ polynomials of degrees $\varepsilon_1+1,\dots,\varepsilon_c+1$, proving the claim.
\end{example}

\begin{remark}
Example \ref{ex:c.i} provide an explicit construction of complete intersections $Z\subset\p^{r+c}$ which are rational Fano varieties of Picard number one and index $r+2-a$ with only $\q$-factorial and terminal singularities. Note that $\Psi:Z\da\p^r$ is the projection from $Y=\p^{c-1}$.
\end{remark}

\begin{remark}
According to Theorem \ref{thm:kleiman} and Remark \ref{rem:porteous}, for generic $u$ in Example \ref{ex:c.i} we get a smooth $X\subset\p^r$ if and only if either $c=1$, or $c\geq 2$ and $r\in\{3,4,5\}$. Furthermore:
\end{remark}

\begin{proposition}\label{prop:c.i}
If $Z\subset\p^{r+c}$ is smooth in Example \ref{ex:c.i} then $Z\subset\p^{r+c}$ is a c.i. $(2,\dots,2)$ and $r\geq 2c-2$. Conversely, in this case, if $u$ is generic and $r\geq 2c-2$ then $Z\subset\p^{r+c}$ is smooth.
\end{proposition}

\begin{proof}
We recall that $\Phi(\Sec_a(X))=Y\subset Z$ is a $\p^{c-1}$, whence $m=c-1$. If $Z$ is smooth then Proposition \ref{prop:num}(ii) yields $r+1=ia-(r-m-1)(ab-1)$. As $b=1$ and $i=r+2-a$, we deduce $m=a-2$ and hence $a=c+1$. In particular, $\sum_{j=1}^c\beta_j=c^2+2c$ so $\beta_j=a+1$ for every $j\in\{1,\dots,c\}$, as $\beta_j\geq a+1=c+2$. Therefore, $Z\subset\p^{r+c}$ is a c.i. $(2,\dots,2)$. The Jacobian matrix of the $c$ equations of $Z\subset\p^{r+c}$ restricted to $Y=\p^{c-1}$ is a $(c,r+1)$-matrix of linear forms in the variables $z_{r+1},\dots,z_{r+c}$ representing a morphism $\O_Y(-1)^{\oplus(r+1)}\to\O_Y^{\oplus c}$. By a well-known extension of Theorem \ref{thm:kleiman}, this matrix drops rank on a subvariety of $Y$ of dimension at least $2c-3-r$. Therefore, $r\geq 2c-2$. Conversely, if $u$ is generic and $r\geq 2c-2$ then the same argument shows that $Z$ is smooth along $Y$, so we deduce that $Z$ is smooth since $\Phi_{|\p^r\backslash\Sec_a(X)}:\p^r\backslash\Sec_a(X)\to Z\backslash Y$ is an isomorphism.
\end{proof}

\begin{remark}
In particular, for generic $u$ in Example \ref{ex:c.i}, we get a special birational transformation if and only if $\beta_j=a+1$ for every $j\in\{1,\dots,c\}$ and either $c=1$ (cf. Corollary \ref{cor:r=n+2}), or $c\geq 2$ and $r\in\{3,4,5\}$ (cf. Theorems \ref{thm:n=1}(VI), \ref{thm:n=2}(IX) and (XII), and \ref{thm:n=3 r=5}(V-VI), respectively).
\end{remark}

\begin{remark}
As $Y\subset Z$ is a linear subspace of dimension $c-1$, we deduce that $Z\subset\p^{r+c}$ needs not to be a general c.i. $(2,\dots,2)$ for $r\geq 2c-2$. On the other hand, we recall that a general c.i. $(2,\dots,2)$ contains a linear subspace of dimension $c-1$ (and hence is rational) for $r\geq(c^2+c-2)/2$ (see \cite{predonzan}).
\end{remark}

A straightforward generalization of \cite[Proposition 4.5]{e-sb} gives a method to produce (special) birational transformations of $\p^r$. It will be used in Examples \ref{ex:g(1,r+1)}, \ref{ex:hypersurfaces} and \ref{ex:degree21} below:

\begin{proposition}\label{prop:E gg}
Let $E$ be a globally generated vector bundle on $\p^r$ of rank $k+1$. Let $p:\p(E)\to\p^r$, and let $q:\p(E)\to\mathcal Z\subset\p(H^0(E))$ be the map corresponding to the tautological line bundle on $\p(E)$, where $\mathcal Z:=q(\p(E))$. If $q:\p(E)\to\mathcal Z$ is a birational map then a general linear subspace $W\subset H^0(E)$ of dimension $k$ gives a birational transformation $\Phi:\p^r\da Z$, where $Z\subset\p^{h^0(E)-1-k}$ is the linear section of $\mathcal Z\subset\p^{h^0(E)-1}$ corresponding to $W$. Moreover, we get the exact sequence $0\to W\otimes\O_{\p^r}\stackrel{u}\to E\to {\mathcal I}_X(c_1(E))\to 0$ and $\Phi$ is given by a basis of $H^0({\mathcal I}_X(c_1(E)))$.
\end{proposition}

\begin{proof}
A general linear subspace $W\subset H^0(E)$ of dimension $k$ corresponds to a general linear subspace $\Sigma_W\subset\p(H^0(E))$ of codimension $k$. In particular, $\Sigma_W$ intersects a general fibre $\p_{P}^k:=q(p^{-1}(P))\subset\mathcal Z$ in a point. Therefore, if $q:\p(E)\to\mathcal Z$ is a birational map then we get a birational transformation $\Phi:\p^r\da Z$ defined by $\Phi(P)=\p_{P}^k\cap\Sigma_W$. Moreover $X:=\{P\in\p^r\mid\dim(\p_{P}^k\cap\Sigma_W)>0\}$, and we get the exact sequence of the statement. Note that $q^{-1}(Z)=\p({\mathcal I}_X(c_1(E)))$, so we deduce that $\Phi$ is given by a basis of $H^0({\mathcal I}_X(c_1(E)))$.
\end{proof}

\begin{remark}
We point out that $q:\p(E)\to\mathcal Z$ is a birational map if and only if $(-1)^rs_r(E)=\deg\mathcal Z$. In particular, $z:=\deg(Z)=(-1)^rs_r(E)$ in Proposition \ref{prop:E gg}.
\end{remark}

\begin{example}\label{ex:g(1,r+1)}
Let $X\subset\p^r$ be given by the resolution
$$0\to\O_{\p^r}^{\oplus r}\to\Omega_{\p^r}(2)\oplus\O_{\p^r}(1)\to {\mathcal I}_X(r)\to 0$$
By Proposition \ref{prop:E gg}, a basis of $H^0({\mathcal I}_X(r))$ gives a birational transformation
$\Phi:\p^r\da Z\subset\p^{\alpha}$ of type $(r,1)$, where ${\alpha}={(r^2+r)/2}$, onto a linear
section of $\g(1,r+1)\subset\p^{(r^2+3r)/2}$. In fact, let us consider the resolution $$0\to\O_{\p^{r+1}}^{\oplus r}\stackrel{u}\to\Omega_{\p^{r+1}}(2)\to{\mathcal I}_{\bar X}(r)\to 0$$
corresponding to a morphism $u$ and defining a subvariety ${\bar X}\subset\p^{r+1}$ of codimension $2$. By arguing as in Example \ref{ex:c.i} and by using the mapping cone technique, we have that a set of $\alpha+1$ generators of degree $r$ of $I_{\bar X}$ is given by the unique syzygy of $M^{t}$, where $M$ is a matrix as follows
\[
M=\left(
\begin{array}{cc}
\begin{array}{c}
\binom{r+2}{2}\times\binom{r+2}{3}\text{ Koszul matrix} \\
\text{$K$ in $r+2$ variables $\underline{x}$}
\end{array}
&
\begin{array}{c}
\binom{r+2}{2}\times r\text{ matrix} \\
\text{of constants}
\end{array}
\end{array}
\right)
\]

Moreover, the first syzygies of this set of generators are the columns of $M$. By choosing coordinates $\underline{x}:=(x_0:\dots:x_{r+1})$ in $\p^{r+1}$ and $\underline{z}:=(z_0:\dots:z_{\alpha})$ in $\p^{\alpha}$, we have that the rational map $\bar{\Phi}$ induced by these generators of $I_{\bar X}$ has a graph given by $\underline{z}M=0$. The general fibre of $\bar{\Phi}$ is an $r$-secant line to ${\bar X}$. A set of generators for $I_{\bar{\Phi}(\p^{r+1})}$ is obtained by eliminating $\underline{x}$ from $\underline{z}M=0$. In this case, $\bar{\Phi}(\p^{r+1})$ is the intersection of $\g(1,r+1)\subset\p^{(r^2+3r)/2}$ with $r$ hyperplanes. By restricting the above resolution to a hyperplane of $\p^r\subset\p^{r+1}$ we get the desired resolution of ${\mathcal I}_X(r)$, where $X:={\bar X}\cap\p^r$ and obviously $Z:=\Phi(\p^r)=\bar{\Phi}(\p^{r+1})$. For generic $u$, the $r$ hyperplanes are general and hence $Z$ is smooth. So we obtain a special birational transformation if and only if $r\in\{3,4,5\}$ by Theorem \ref{thm:kleiman} and Remark \ref{rem:porteous} (cf. Theorems \ref{thm:n=1}(VII), \ref{thm:n=2}(X) and \ref{thm:n=3 r=5}(IV), respectively).
\end{example}

In a slightly different way we obtain the following:

\begin{example}\label{ex:g(1,r)}
Let $X\subset\p^r$ be given by the resolution
$$0\to\O_{\p^r}^{\oplus r-2}\oplus\O_{\p^r}(-1)\stackrel{u}\to\Omega_{\p^r}(2)\to {\mathcal I}_X(r)\to 0$$
corresponding to a morphism $u$. Then a basis of $H^0({\mathcal I}_X(r))$ gives a birational transformation $\Phi:\p^r\da Z\subset\p^{\alpha}$ of type $(r,2)$, where ${\alpha}={(r^2-r+2)/2}$, onto a linear
section of $\g(1,r)\subset\p^{(r^2+r-2)/2}$. In fact, arguing as in Examples \ref{ex:c.i} and \ref{ex:g(1,r+1)}, we get
\[
M=\left(
\begin{array}{ccc}
\begin{array}{c}
\binom{r+1}{2}\times\binom{r+1}{3}\\
\text{Koszul matrix $K$}\\
\text{in $r+1$ variables $\underline{x}$}
\end{array}
&
\begin{array}{c}
\binom{r+1}{2}\times 1\\
\text{matrix of}\\
\text{linear forms}
\end{array}
&
\begin{array}{c}
\binom{r+1}{2}\times (r-2)\\
\text{matrix of}\\
\text{constants}
\end{array}
\end{array}
\right)
\]
and it is easy to see from the syzygies of the ideal that $\Phi$ is a birational transformation and that $Z$ is contained in a linear section as above. Having the same dimension, they coincide. For generic $u$, the $r-2$ hyperplanes are general and hence $Z$ is smooth. So we obtain a special birational transformation if and only if $r\in\{3,4,5\}$ by Theorem \ref{thm:kleiman} and Remark \ref{rem:porteous} (cf. Theorems \ref{thm:n=1}(V), \ref{thm:n=2}(VIII) and \ref{thm:n=3 r=5}(II), respectively).
\end{example}

Before presenting the next example, we need to recall some known facts (see \cite{e-sb}, pp.~799--800). Let $A$ denote the vector space of skew-symmetric matrices $(r+1)\times(r+1)$, and let $A_k:=\{a\in A\mid\rk(a)\leq k\}$. The rank of a skew-symmetric matrix is an even number, and if $k$ is even then $A_k$ has codimension $\binom{r+1-k}{2}$ in $A$ and $\sing(A_k)=A_{k-2}$. Let $E:=\Lambda^2(T_{\p^r}(-1))\cong\Omega^{r-2}_{\p^r}(r-1)$. Then $E$ is
generated by its global sections, that can be identified with $A$. Let $p:\p(E)\to\p^r$ and $q:\p(E)\to\p(A)$ denote the projection maps, with $q$ given by the tautological line bundle on $\p(E)$. We can distinguish two cases. If $r=2t$ then $s_{2t}(E)=1$ and $q:\p(E)\to\p(A)$ is a birational map. On the other hand, if $r=2t-1$ then $s_{2t-1}(E)=0$ and $q(\p(E))=:{\mathcal Z}\subset\p(A)$ is the Pfaffian hypersurface of degree $t$, corresponding to $A_{r-1}$, whose singular locus, corresponding to $A_{r-3}$, has codimension $6$ in $\p(A)$. Now consider the resolution
$$0\to\O_{\p^r}^{\oplus (r^2-r-2)/2}\stackrel{u}\to\Lambda^2(T_{\p^r}(-1))\to {\mathcal I}_{\bar X}(r-1)\to 0$$
corresponding to a morphism $u$. If $r=2t$ then a basis of $H^0({\mathcal I}_{\bar X}(r-1))$ gives a Cremona transformation of type $(2t-1,t)$ which is special for generic $u$ only
for $t=2$ by Theorem \ref{thm:kleiman} and Remark \ref{rem:porteous} (in that case ${\bar X}\subset\p^4$ is a quintic elliptic scroll). The fundamental locus of the inverse map is given by the
intersection of the Pfaffian hypersurfaces of degree $t$ and $(r^2-r-2)/2$ hyperplanes (see \cite[p.~798]{e-sb}). On the other hand, if $r=2t-1$ then a basis of
$H^0({\mathcal I}_{\bar X}(r-1))$ gives a subhomaloidal system $\p^r\da Z\subset\p^r$, whose general fibre is a line, onto a hypersurface of degree $t$. Note that $Z\subset\p^r$ is the intersection of the Pfaffian hypersurface ${\mathcal Z}\subset\p(A)$ of degree $t$ with $(r^2-r-2)/2$ hyperplanes.

\begin{example}\label{ex:hypersurfaces}
Restricting the above resolution to a hyperplane $\p^{r-1}\subset\p^r$, we obtain the resolution
$$0\to\O_{\p^{r-1}}^{\oplus (r^2-r-2)/2}\to\Lambda^2(T_{\p^{r-1}}(-1))\oplus T_{\p^{r-1}}(-1)\to {\mathcal I}_X(r-1)\to 0$$
where $X:={\bar X}\cap\p^{r-1}$. By Proposition \ref{prop:E gg} and the above discussion, we get a birational transformation $\Phi:\p^{r-1}\da
Z\subset\p^r$ of type $(r-1,\lceil(r-1)/2)\rceil$ onto a hypersurface of degree $\lceil r/2\rceil$ having in mind that $b=\lceil(r-1)/2)\rceil$ by Proposition \ref{prop:num}, as $i=r+1-\lceil r/2\rceil$ since $Z\subset\p^r$ is a hypersurface of degree $\lceil r/2\rceil$. We remark that $\sing(\mathcal Z)\subset\p(A)$ has codimension $6$, and hence $\dim(\sing(Z))\geq r-6$. If $u$ is generic then the $(r^2-r-2)/2$ hyperplanes are general. So $Z$ is smooth if and only if $r\in\{4,5\}$, and hence $\Phi:\p^{r-1}\da Z\subset\p^r$ is a special birational transformation onto a prime Fano manifold by Theorem \ref{thm:kleiman} and Remark \ref{rem:porteous} (cf. Theorems \ref{thm:n=1}(V) and \ref{thm:n=2}(VII), respectively).
\end{example}

\begin{remark}\label{rem:hypersurfaces}
The resolution of Example \ref{ex:hypersurfaces} is equivalent, via Euler's exact sequence, to the resolution
$$0\to\O_{\p^{r-1}}^{\oplus (r^2-3r-2)/2}\oplus\O_{\p^{r-1}}(-1)\to\Lambda^2(T_{\p^{r-1}}(-1))\to {\mathcal I}_X(r-1)\to 0$$
And this is also equivalent to the resolution
$$0\to T_{\p^{r-1}}(-2)\oplus\O_{\p^{r-1}}(-1)\to\O_{\p^{r-1}}^{\oplus r+1}\to {\mathcal I}_X(r-1)\to 0$$
by means of the sequence $0\to T_{\p^{r-1}}(-2)\to\O_{\p^{r-1}}^{\oplus (r^2-r)/2}\to\Lambda^2(T_{\p^{r-1}}(-1))\to 0$.
\end{remark}

\begin{example}\label{ex:degree21}
Let $F:=E\oplus\O_{\p^r}(1)$, with $E=\Lambda^2(T_{\p^r}(-1))$. Let $p:\p(F)\to\p^r$ and $q:\p(F)\to\p(A\oplus\c^{r+1})$ denote the projection maps, with $q$ given by the tautological
line bundle on $\p(F)$. We point out that $q$ is a birational map onto $\mathcal Z:=q(\p(F))$ whose fundamental locus is contained in $q^{-1}(\mathcal Z\cap\p(A))$. In particular, $\sing(\mathcal Z)\subset\p(A)$. Let us compute the dimension of $\sing(\mathcal Z)$. If $r=2t$ then $\sing(\mathcal Z)$ corresponds to $A_{2t-2}$, and hence $\sing(\mathcal Z)\subset\p(A)$ has codimension $3$. On the other hand, if $r=2t-1$ then $\sing(\mathcal Z)$ corresponds to $A_{2t-4}$, and hence $\sing(\mathcal Z)\subset\p(A)$ has codimension $6$. Now we apply Proposition \ref{prop:E gg}. Consider the resolution $$0\to\O_{\p^r}^{\oplus (r^2-r)/2}\stackrel{u}\to\Lambda^2(T_{\p^r}(-1))\oplus\O_{\p^r}(1)\to {\mathcal I}_X(r)\to 0$$
corresponding to a morphism $u$. As $h^0({\mathcal I}_X(r-1))=1$, we deduce from Proposition \ref{prop:sec} that $b=1$. If $r=2t$ then we get a birational transformation $\p^r\da Z\subset\p^{2r}$ of type $(r,1)$ and $\dim(\sing(Z))\geq r-4$. On the other hand, if $r=2t-1$ then we get a birational transformation $\p^r\da Z\subset\p^{2r}$ of type $(r,1)$ and $\dim(\sing(Z))\geq r-7$. Note that $Z\subset\p^{2r}$ is the intersection of $\mathcal Z\subset\p(A\oplus\c^{r+1})$ with $(r^2-r)/2$ hyperplanes. If $u$ is generic then the $(r^2-r)/2$ hyperplanes are general. So $Z$ is smooth if and only if $r\in\{3,5\}$, and hence $\Phi:\p^r\da Z\subset\p^{2r}$ is a special birational transformation onto a prime Fano manifold by Theorem \ref{thm:kleiman} and Remark \ref{rem:porteous} (cf. Theorems \ref{thm:n=1}(VII) and \ref{thm:n=3 r=5}(III), respectively). Equivalently, ${\mathcal I}_X$ is given by the resolution
$0\to T_{\p^r}(-2)\to\O_{\p^r}^{\oplus r}\oplus\O_{\p^r}(1)\to {\mathcal I}_X(r)\to 0$.
\end{example}

\begin{remark}
If $r=3$ then $Z\subset\p^6$ is a linear section of $\mathcal Z=\g(1,4)\subset\p^9$. On the other hand, if $r=5$ then $Z\subset\p^{10}$ is a linear section of $\mathcal Z\subset\p^{20}$ and $\dim(\sing(\mathcal Z))=8$. Therefore, $Z\subset\p^{10}$ is a hyperplane section of a $6$-dimensional prime Fano manifold in $\p^{11}$ of coindex $4$ and degree $21$. To the best of the authors' knowledge, there is no reference to this manifold in the literature. Finally, for $r=4$ we get a small contraction onto $Z\subset\p^8$ as $\Sec_4(X)$ is the union of $5$ planes in $\p^4$ and not a cubic hypersurface.
\end{remark}

\begin{example}\label{ex:quadrics}
Let us collect some well-known examples of special birational transformations of $\p^r$ defined by quadric hypersurfaces:
\begin{enumerate}
\item[(i)] (see \cite[Theorem 3.8.3]{zak2}) Let $X=\p^1\times\p^{n-1}\subset\p^{2n-1}\subset\p^{2n}$ be a degenerate Segre embedding. Then a basis of $H^0({\mathcal I}_X(2))$ gives a special birational transformation $\Phi:\p^{2n}\da\g(1,n+1)\subset\p^{(n^2+3n)/2}$ of type $(2,1)$.
\item[(ii)] (see \cite[Theorem 3.8.5]{zak2}) Let $X=\g(1,4)\subset\p^9\subset\p^{10}$ be a degenerate embedding. Then a basis of $H^0({\mathcal I}_X(2))$ gives a special birational transformation $\Phi:\p^{10}\da S_4\subset\p^{15}$ of type $(2,1)$ onto the $10$-dimensional spinor variety.
\item[(iii)] (see \cite{semple}) Let $X\subset\p^{2n+2}$ be a rational normal scroll. Then a basis of $H^0({\mathcal I}_X(2))$ gives a special birational transformation $\Phi:\p^{2n+2}\da\g(1,n+2)\subset\p^{(n^2+5n+4)/2}$ of type $(2,2)$.
\item[(iv)] (see \cite[Theorem 2.6]{e-sb}) Let $X\subset\p^r$ be a Severi variety. Then a basis of $H^0({\mathcal I}_X(2))$ gives a special Cremona transformation $\Phi:\p^r\da\p^r$ of type $(2,2)$. So the restriction to a general hyperplane gives a special bitrational transformation $\Phi:\p^{r-1}\da Z\subset\p^r$ of type $(2,2)$ onto a smooth quadric hypersurface.
\end{enumerate}
\end{example}

\begin{remark}\label{rem:cremona}
The series of Cremona transformations that appear in this paper are:
\begin{enumerate}
\item[(i)] The Cremona transformation $\Phi:\p^{2n+2}\da\p^{2n+2}$ given by the quadric hypersurfaces containing a generic elliptic scroll of degree $2n+3$ (see \cite{s-t2}).
\item[(ii)] The Cremona transformation $\Phi:\p^r\da\p^r$ given by $r\times r$-minors of an $r\times (r+1)$-matrix of linear forms (see \cite{e-sb}).
\item[(iii)] The Cremona transformations $\Phi:\p^{kd}\da\p^{kd}$ given by forms of degree $d$ described in \cite[Section 5]{h-k-s}.
\end{enumerate}
\end{remark}

\bibliography{bibfile}
\bibliographystyle{amsplain}

\end{document}